\documentclass[10pt,oneside]{amsart}
\usepackage{latexsym}
\usepackage{amsfonts,amsmath,amssymb,indentfirst,color,tikz}
\usepackage[english]{babel}
\usepackage[all]{xy}
\usepackage[pdftex]{hyperref}

\newcommand{\bdism}{\begin{displaymath}}
\newcommand{\edism}{\end{displaymath}}
\newcommand{\cc}{\mathbb{C}}

\newcommand{\qq}{\mathbb{Q}}
\newcommand{\zz}{\mathbb{Z}}

\newcommand{\pp}{\mathbb{P}}

\DeclareMathOperator{\PU}{\mathrm{PU}}

\newtheorem{theorem}{Theorem}[section]
\newtheorem{proposition}[theorem]{Proposition}
\newtheorem{corollary}[theorem]{Corollary}
\newtheorem{lemma}[theorem]{Lemma}
\newtheorem{remark}[theorem]{Remark}

\newtheorem{fact}[theorem]{Fact}
\newtheorem*{claim}{Claim}

\address{Mathematics Section - 11 Strada Costiera, ICTP, Trieste 34151, Italy} \email{ldicerbo@ictp.it}
\address{Department of Mathematics 1805 N. Broad Street, Philadelphia, PA 10122, USA} \email{mstover@temple.edu}

\author{Luca F. Di Cerbo, Matthew Stover}

\title{\bf Classification and arithmeticity of toroidal compactifications with $3\overline{c}_{2}=\overline{c}^{2}_{1} = 3$}

\begin{document}

\begin{abstract}
We classify the minimum volume smooth complex hyperbolic surfaces that admit smooth toroidal compactifications, and we explicitly construct their compactifications. There are five such surfaces and they are all arithmetic, i.e., they are associated with quotients of the ball by an arithmetic lattice. Moreover, the associated lattices are all commensurable. The first compactification, originally discovered by Hirzebruch, is the blowup of an Abelian surface at one point. The others are bielliptic surfaces blown up at one point. The bielliptic examples are new, and are the first known examples of smooth toroidal compactifications birational to bielliptic surfaces.
\end{abstract}
\maketitle

\tableofcontents

\markboth{Right}{\textbf{CLASSIFICATION OF TOROIDAL COMPACTIFICATIONS}}

\section{Introduction}

\pagenumbering{arabic}

A classical and important problem in algebraic geometry is to classify surfaces of general type with given numerical invariants. For some aspects of this fascinating and long-standing problem we refer to the recent survey \cite{BCP}. A situation that has attracted much recent interest is the case $c_1^2 = 3 c_2 = 9$, where $c_1^2$ is the self-intersection of the canonical bundle and $c_2$ the Euler number of the underlying surface. This problem is particularly interesting because of its connection with low-dimensional geometry. In fact, Hirzebruch proportionality, Yau's solution to the Calabi conjecture \cite{Yau} and the work of Miyaoka \cite{Miyaoka} imply that this is equivalent to classifying the minimum volume quotients of the unit ball in $\cc^{2}$ by a torsion-free cocompact lattice in $\PU(2,1)$. Most famously, this includes the classification of fake projective planes by Prasad--Yeung \cite{Pra} and Cartwright--Steger \cite{Ste}. Recall that Mumford constructed the first example of a fake projective plane by $p$-adic uniformization and explicitly raised the problem of classifying them \cite{Mum}.

This paper considers the corresponding problem in the noncompact or logarithmic setting. For example, given $r, s$, one can ask for a classification of the smooth projective surfaces $X$ containing a normal crossings divisor $D$ for which the logarithmic Chern numbers satisfy
\[
\overline{c}^2_1(X, D) = r \quad\quad \overline{c}_2(X, D) = s,
\]
where $\overline{c}^2_1(X, D)$ is the self-intersection of the log canonical divisor $K_X + D$ and $\overline{c}_2(X, D)$ is the Euler number of $X \smallsetminus D$. We refer to \cite{Sakai} for the foundation of the theory of logarithmic surfaces. For some more recent results see \cite{Urzua}. In this setting, $\overline{c}^{2}_{1} = 3 \overline{c}_{2}$ now implies that $X \smallsetminus D$ is a noncompact finite volume quotient of the ball by a torsion-free lattice \cite{TianY}, and $X$ is then a smooth toroidal compactification of a ball quotient. In fact, Tian--Yau derive a logarithmic Bogomolov--Miyaoka--Yau inequality for surfaces of log-general type and then characterize the pairs that attain equality as smooth toroidal compactifications of ball quotients. 

In this paper, we classify all smooth toroidal compactifications with $\overline{c}^{2}_{1} = 3 \overline{c}_{2} = 3$. Equivalently, we classify the minimum volume complex hyperbolic surfaces that admit smooth toroidal compactifications. Our main result is the following.

\begin{theorem}\label{primo}
There are exactly five toroidal compactifications with $3\overline{c}_{2}=\overline{c}^{2}_{1} = 3$. One has underlying space an Abelian variety blown up at one point. The other four have underlying space a bielliptic surface blown up once. The associated lattices in $\PU(2,1)$ are arithmetic and commensurable.
\end{theorem}

The blown up Abelian surface in Theorem \ref{primo} was first studied by Hirzebruch in \cite{Hir84}, and later shown to be the compactification of a Picard modular surface by Holzapfel \cite{Holzapfel}. The other four bielliptic examples are new, and are the first examples of smooth toroidal compactifications of ball quotients birational to bielliptic surfaces. All previously known examples not of general type (Hirzebruch's \cite{Hir84} and an additional example found by Holzapfel \cite{Holz04}) are birational to an Abelian surface. They are also the first completely explicit examples, in the sense that they are complements of a specific divisor on a given surface, not of Abelian type.

In contrast, we note that it is one of the central results in the study of toroidal compactifications that any lattice $\Gamma$ in $\PU(2,1)$ contains a subgroup $\Gamma^\prime$ of finite index such that the associated ball quotient admits a smooth toroidal compactification of general type \cite[Thm.\ IV.1.4]{Ash}. This result has been refined in \cite{Luca1} where it is shown that, up to finite covers, all such compactifications have ample canonical divisor. In fact, it was previously believed that any smooth toroidal compactification not of general type must be birational to an Abelian surface (see \cite{Momot}, which unfortunately contains a critical error), and our results explicitly refute this. For applications of the existence of bielliptic smooth toroidal compactifications to problems on volumes of complex hyperbolic manifolds and group-theoretic properties of lattices in $\PU(2, 1)$, see \cite{DS16}.

In light of the analytical results contained in \cite{TianY}, Theorem \ref{primo} implies the following result of interest for the geography of surfaces of log-general type.

\begin{corollary}
There are exactly five surfaces $(X, D)$ of log-general type satisfying $3\overline{c}_{2}=\overline{c}^{2}_{1} = 3$.
\end{corollary}

Theorem \ref{primo} has also the following corollary regarding the arithmeticity of the fundamental groups of minimum volume ball quotients. Recall that a noncompact ball quotient with finite volume admits a smooth toroidal compactification when the parabolic elements of its fundamental group have no rotational part (see \S \ref{prelim}).

\begin{corollary}\label{conjecture}
Let $\Gamma$ be a torsion-free lattice in $\PU(2,1)$ with minimum covolume. If the parabolic elements in $\Gamma$ have no rotational part, then $\Gamma$ is arithmetic.
\end{corollary}

This gives further evidence toward the folklore conjecture that minimum volume locally symmetric manifolds and orbifolds are arithmetic (see, e.g., \cite{Belo}). More precisely, Corollary \ref{conjecture} proves this conjecture for torsion-free nonuniform lattices in $\PU(2,1)$ with rotation-free parabolic elements. In this language, Theorem \ref{primo} contributes to the wide literature on classification of minimal covolume lattices in semisimple Lie groups. For example, see the important very recent work of Gabai--Meyerhoff--Milley \cite[Cor.\ 1.2]{GMM} for the solution to the analogous problem for cusped hyperbolic $3$-manifolds.\\

The proofs of the above results follow an algebro-geometric approach. More precisely, we fully exploit the implications of the Kodaira--Enriques classification for smooth toroidal compactifications of ball quotients with Euler number one, following a program outlined by the first author in \cite{DiC13}. All five examples are associated with arithmetic lattices and they appear in the appendix to \cite{Stover}, where the second author gave several examples of noncompact arithmetic ball quotients of Euler number one. We note that in \cite{Stover} there are three other ball quotients of Euler number one that do not admit smooth toroidal compactifications. See Section \ref{thmproof} for more on this.

We briefly touch upon an analogy between our work and the classification of fake projective planes \cite{Kli, Pra, Ste}. On the one hand, fake projective planes are the minimum volume closed complex hyperbolic $2$-manifolds with first betti number zero. On the other hand, they are precisely the minimal surfaces of general type with irregularity $0$ and $c_1^2 = 3 c_2 = 9$. The techniques used in this paper are different from those used in the classification of fake projective planes. Crucial to the classification of fake projective planes is the proof that their fundamental groups are arithmetic \cite{Kli, Yeu}. Recall that nonarithmetic lattices are known to exist by the work of Mostow and Deligne--Mostow (see \cite{Deligne}). From that point, fake projective planes are then classified by enumerating the torsion-free arithmetic lattices in $\PU(2,1)$ of the appropriate covolume. Our proof is of a completely different nature. We first classify the possible smooth toroidal compactifications using algebro-geometric techniques, then deduce arithmeticity from commensurability with Hirzebruch's ball quotient.

We now outline the organization of the paper. Section \ref{prelim} starts with a short review of the geometry of smooth toroidal compactifications. The problem of finding the smallest toroidal compactifications is formulated in terms of logarithmic Chern numbers. In Section \ref{different}, we show that a toroidal compactification with $\overline{c}_{2}=1$ cannot have Kodaira dimension two, one, or $-\infty$. In light of the Kodaira--Enriques classification, this reduces the problem to the Kodaira dimension zero case.

In Section \ref{zero}, we study the Kodaira dimension zero case in detail. It is shown that the minimal model of a toroidal compactification with $\overline{c}_{2}=1$ that is not a bielliptic surface must be the product of two elliptic curves with large automorphism group. Finally, an explicit example is constructed and its uniqueness is proved.

In Section \ref{Bagnera}, we then classify all toroidal compactifications with $\overline{c}_{2}=1$ that are birational to a bielliptic surface. There are exactly four of them. We briefly recap the constructions of the five examples in \S \ref{recap}, and in \S \ref{thmproof} we show that they are commensurable and arithmetic.

\vspace{0.5cm}

\noindent\textbf{Acknowledgments}. The first author would like to thank the Mathematics Department at Notre Dame University for the nice research environment he enjoyed while carrying out most of this project. The first author was supported by the S.-S. Chern grant at ICTP. The second author was supported by the National Science Foundation under Grant Number NSF 1361000. The second author also acknowledges support from U.S. National Science Foundation grants DMS 1107452, 1107263, 1107367 "RNMS: GEometric structures And Representation varieties" (the GEAR Network). The authors also thank the referee for a very careful reading of the paper.

\section{Preliminaries}\label{prelim}

The theory of compactifications of locally symmetric varieties has been extensively studied. For example, see \cite{Borel2}. Let $\mathcal{H}^{n}$ be $n$-dimensional complex hyperbolic space. Noncompact finite volume complex hyperbolic manifolds then correspond with conjugacy classes of torsion-free nonuniform lattices in $\PU(n,1)$. Let $\Gamma$ be any such lattice. It is well-known that when the parabolic elements in $\Gamma$ have no rotational part, the manifold $\mathcal{H}^{n}/\Gamma$ has a particularly nice compactification $(X, D)$ consisting of a smooth projective variety $X$ and boundary divisor $D$. In other words, we require the subgroup of $\cc$ generated by the eigenvalues of parabolic elements of $\Gamma$ to be torsion-free. Under these assumptions the divisor $D$ is the union of smooth disjoint Abelian varieties, each having normal bundle of negative degree. The pair $(X, D)$ is referred as a \emph{toroidal} compactification of $\mathcal{H}^{n}/\Gamma$. For more details about this construction, see \cite{Hummel}, \cite{Ash}, and \cite{Mok}. Note that in \cite{Mok} this construction is carried out without any arithmeticity assumption on $\Gamma$.

We now describe the two dimensional case in more detail. Let $\mathcal{H}^{2}/\Gamma$ be a complex hyperbolic surface with cusps that admits a smooth toroidal compactification $(X, D)$. It is known (see \cite{Sakai} and \cite[\S 4]{Luca1}) that $(X, D)$ is $D$-minimal of log-general type, where $(X, D)$ is $D$-minimal if $X$ does not contain any exceptional curve $E$ of the first kind such that $D\cdot E\leq 1$. Moreover, by the Hirzebruch--Mumford proportionality principle \cite{Mumford} we have
\[
3\overline{c}_{2}=\overline{c}^{2}_{1},
\]
where $\overline{c}_{1}$ and $\overline{c}_{2}$ are the logarithmic Chern numbers of the pair $(X, D)$.

For the standard properties of logarithmic Chern classes we refer to \cite{Kaw}. Recall that $\overline{c}^{2}_{1}$ is equal to the self-intersection of the log-canonical divisor $K_{X}+D$, while $\overline{c}_{2}$ is simply the topological Euler characteristic of $X\smallsetminus D$. Since $D$ consists of smooth disjoint elliptic
curves, we have
\[
\overline{c}_{2}(X)=\chi(X)-\chi(D)=\chi(X)=c_{2}(X).
\]
By construction, $X\smallsetminus D$ is equipped with a complete metric with pinched negative sectional curvature. For this class of metrics it is well known that the Pfaffian of the curvature matrix is point-wise strictly positive (see \cite[\S 2]{Luca1}). Thus, Gromov--Harder's generalization of Gauss--Bonnet \cite{Gro1} implies that $\overline{c}_{2}$ must be a strictly positive integer. In particular, we see that $\overline{c}_{2}=1$ is the minimum possible value, and hence the Chern--Gauss--Bonnet theorem implies that the smooth toroidal compactifications with $3\overline{c}_{2}=\overline{c}^{2}_{1} = 3$ are minimum volume complex hyperbolic $2$-manifolds.

We close this section with the following fact that will be useful many times in the paper.

\begin{theorem}[Thm.\ 3.18 \cite{DiCerbo2}]\label{thm:CuspCount}
Let $(X, D)$ be a toroidal compactification of dimension $n \ge 2$. Let $q$ be the number of components of $D$. Then $q < \rho(X)$, where $\rho(X)$ is the Picard number of $X$.
\end{theorem}

In particular, the Picard number of $X$ strictly bounds the number of cusps from above. We briefly give the argument. The divisor $D$ gives $q$ disjoint elliptic curves of negative self-intersection, which, along with the class of an ample divisor, generate a subspace of dimension $q + 1$ in the N\'eron--Severi group of $X$, and the result follows.

\section{The case $\kappa\neq$ 0}\label{different}

In this section we show that a toroidal compactification $(X, D)$ with $\overline{c}_{2}=1$ must have Kodaira dimension $\kappa(X)$ equal to zero. Let us start by showing that $X$ cannot be of general type.

\begin{lemma}
Let $(X, D)$ be a toroidal compactification with $\overline{c}_{2}=1$. Then $X$ cannot have $\kappa(X)=2$.
\end{lemma}

\begin{proof}
Assume for a contradiction that such a pair $(X, D)$ exists. Recall that, given a surface $Y$ and $Bl_{k}(Y)$ the blowup of $Y$ at $k$ points, the second Chern number of $Bl_{k}(Y)$ is given by
\[
c_{2}(Bl_{k}(Y)) = c_{2}(Y) + k.
\]
It is well known that the Euler number of a minimal surface of general type is strictly positive (see \cite[Prop.\ VII.2.4]{Bar}). Since $c_{2}(X)=1$, we conclude that $X$ must be minimal.

Next, observe that the adjunction formula implies that $K_X \cdot D = -D^2$, and so
\[
\overline{c}^{2}_{1}=c^{2}_{1}-D^{2}=3\overline{c}_{2}=3c_{2},
\]
which implies that
\begin{align}\notag
0<c^{2}_{1}<3c_{2},
\end{align}
since $D^{2}<0$ and $c^{2}_{1}>0$ for any minimal surface of general type. We then have $c^{2}_{1}\in\{1, 2\}$. However, for any complex surface we must have
\begin{align}\notag
c^{2}_{1}+c_{2} \equiv 0\pmod{12}
\end{align}
by Noether's formula \cite[p.\ 9]{Friedman}. We therefore conclude that $(X, D)$ cannot have $X$ be of general type.
\end{proof}

Now we rule out the case of Kodaira dimension one.

\begin{lemma}\label{lem:notk=1}
Let $(X, D)$ be a toroidal compactification with $\overline{c}_{2}=1$. Then $X$ cannot have $\kappa(X)=1$.
\end{lemma}

\begin{proof}
Given $X$, there exists a unique minimal model $Y$ such that $c^{2}_{1}(Y)=0$ and $c_{2}(Y)\geq 0$ \cite[Thm.\ VI.1.1]{Bar}. Noether's formula then implies that
\[
c_{2}(Y)=12d
\]
for some $d\in\zz_{\geq0}$. It follows that a surface with $\kappa(X)=1$ must satisfy
\[
c_{2}(X)=12d+k
\]
with $d, k\in\zz_{\geq0}$. Therefore, if we want $c_{2}(X)=1$, we must have $d=0$ and $k=1$. In other words, $X$ is the blowup at just one point of a minimal elliptic surface $Y$ with Euler number zero.

For a minimal elliptic fibration
\[
\pi: Y \longrightarrow B
\]
with multiple fibers $F_{1}, ..., F_{k}$ of multiplicities $m_{1}, \dots, m_{k}$ we have
\begin{equation}\label{canonicalf}
K_{Y}=\pi^{*}(K_B \otimes L) \otimes \mathcal{O}_{Y}(\sum^{k}_{i=1}(m_{1}-1)F_{i})
\end{equation}
where $L=(R^{1}\pi_{*}\mathcal{O}_{Y})^{-1}$. Note that $d=\mathrm{deg}(L)$ (see \cite[Cor.\ 16, p.\ 177]{Friedman}). In the case under consideration $c_2(Y) = 0$, so all the singular fibers of the elliptic fibration are multiple fibers with smooth reduction \cite[Cor.\ 17, p.\ 177]{Friedman}.

Now, consider
\[
f: X\longrightarrow B,
\]
where $f=\pi\circ Bl$ and $Bl: X\longrightarrow Y$ is the blowup map. We claim that some irreducible component $D_i$ of $D$ must map onto $B$ under $f$. If every $D_i$ were contained in a fiber, then there would be a general fiber of $f$ that does not meet the divisor $D$, which means that there would exist an irreducible smooth elliptic curve $E$ in $X \smallsetminus D$. The existence of such a holomorphic curve $E$  is impossible because $X \smallsetminus D$ is by construction hyperbolic. This proves the claim.

Since $f(D_{i}) = B$ for some $i$ and $D_i$ is an elliptic curve, the Hurwitz formula then implies that the genus of $B$ must be $0$ or $1$. Indeed, $D_i$ cannot be contained in a fiber of the fibration, since fibers are rational. Therefore the elliptic curve $D_i$ maps onto $B$, and so $B$ must have genus at most one. Equation \eqref{canonicalf} implies that if $\kappa(Y)=1$, we must assume the existence of multiple fibers. Indeed, otherwise $\kappa(Y) = -\infty$ when $g(B) = 0$ and $\kappa(Y) = 0$ when $g(B) = 1$ (see \cite[\S V.12]{Bar}).

Let $(Y, C)$ be the blow down configuration of $(X, D)$. We first study the case $g(B)=1$. We then have that some irreducible component $C_i$ of $C$ is a holomorphic $n$-section of the elliptic fibration, i.e., the map $C_i \to B$ is generically $n$-to-$1$. Moreover, $C_i$ is normalized by an irreducible component $D_i$ of $D$, which is a smooth elliptic curve. Let $Y'$ be the fiber product $Y\times_B D_i$, so $Y' \longrightarrow Y$ is an \'etale covering of degree $n$. However, $Y'$ then has a holomorphic $1$-section, which implies that $\kappa(Y')=0$, since there cannot be multiple fibers and all fibers are nonsingular. Then $\kappa(Y) = 0$ by invariance of $\kappa$ under unramified coverings, which is a contradiction.

Now assume that $g(B)=0$. In this case $\mathrm{deg}(L)=0$, so $L$ is trivial. Again, there is a holomorphic $n$-section $C_i$ that is normalized by a smooth elliptic curve $D_i$. Following a base change argument as in the setup for the last exercise on \cite[p.\ 193]{Friedman}, there is a finite unramified cover $\pi': Y'\longrightarrow D_i$ of $\pi:Y\longrightarrow B$ with a holomorphic section and such that $L'=\mathcal{O}_{D_i}$. We then have that $Y'=D_i\times E$ for some elliptic curve $E$. By \cite[Ch.\ VI, \S 1]{Bar}, the Kodaira dimension of $Y$ cannot be one.
\end{proof}

We now show that $X$ cannot be birational to a rational or ruled surface.

\begin{lemma}\label{lem:NotRational}
Let $(X, D)$ be a toroidal compactification with $\overline{c}_{2}=1$. Then $X$ cannot have $\kappa(X)=-\infty$.
\end{lemma}

\begin{proof}
Recall that the minimal models of surfaces with negative Kodaira dimension are $\pp^{2}$, the Hirzebruch surfaces $X_{e}$, and ruled surfaces over Riemann surfaces of genus $g\geq 1$ \cite[Ch.\ VI]{Bar}. Since
\[
c_{2}(\pp^{2})=3,\quad c_{2}(X_{e})=4,
\]
and $c_2$ only increases under blowup, a toroidal compactification with $\overline{c}_{2}=1$ must have minimal model a ruled surface with base of genus $g \ge 1$.

Then $c_2(X) = 4(1 - g) + k$, where $k$ is the number of blowup points, so $X$ must be the blowup at exactly one point of a surface $Y$ ruled over an elliptic curve. Indeed, similar to the argument in Lemma \ref{lem:notk=1}, any elliptic curve $D_i$ in the compactification must map onto the base, which has genus $g \ge 1$, and hence $g = 1$. This fact combined with the formula for $c_{2}(X)$ given above implies that $k=1$. Therefore, $c_{2}(Y)=0$. Since the rank of the Picard group of $X$ is three, by Theorem \ref{thm:CuspCount} we have that $X$ can have at most two cusps.

Let $(Y, C)$ denote the blow down configuration of $(X, D)$, and assume that $C$ consists of exactly one irreducible component. We then have by an argument similar to the $\kappa = 1$ case that $C$ must be an $n$-section of the ruling of $Y$. It is easily seen that $C$ cannot be a smooth $n$-section of the ruling for any $n \ge 1$, since this implies that $X \smallsetminus D$ contains a $\pp^{1}$ with just one puncture, namely the exceptional fiber of the blowup, which contradicts hyperbolicity of the metric on $X \smallsetminus D$. This argument implies that $C$ must be singular at some point $p$. In fact, since $\overline{c}_2(X) = 1$, $p$ must be the unique singular point, and hence $X$ is the blowup of $Y$ at $p$.

Consider the composition of the blowdown $X \to Y$ with the map of $Y$ to the base $B$ of the ruling. Then $D$ is a smooth elliptic curve on $X$ not contained in a fiber of the map to $B$, hence $D \to B$ is a surjective map between elliptic curves. Such a map must be \'etale, but it factors through the map from $D$ onto the singular curve $C$. This is a contradiction.

Let us conclude by studying the case when $(X, D)$ has two cusps. In this situation $(Y, C)$ is such that $C$ consists of two irreducible components, say $C_{1}$ and $C_{2}$, intersecting in a point $p$, where $p$ is the point of $Y$ blown up to obtain $X$. Considering the exceptional divisor of the blowup, if both $C_1$ and $C_2$ were smooth multisections then we would obtain a twice-punctured $\pp^1$ in $X \smallsetminus D$, which is impossible. They are also clearly not sections of the fibration, since their proper transform to $X$ would not be an elliptic curve. It follows that at least one of $C_1$ and $C_2$ must be singular at the blowup point $p$. Moreover, the tangents lines of $C_{1}$ and $C_{2}$ at $p$ must be all distinct for the proper transforms to be disjoint. We can then proceed as in the one cusp case to obtain a contradiction.
\end{proof}

We summarize the results of this section as a proposition.

\begin{proposition}\label{reduction}
Let $(X, D)$ be a toroidal compactification with $\overline{c}_{2}=1$. Then the Kodaira dimension of $X$ is zero.
\end{proposition}

Of course, it remains to be seen if any such example actually exists. The next two sections completely solve this problem.

\section{The case $\kappa$ = 0}\label{zero}

In light of Proposition \ref{reduction}, a toroidal compactification with $\overline{c}_{2}=1$ must be birational to a minimal surface of Kodaira dimension zero. Recall that minimal surfaces with Kodaira dimension zero are given by:
\begin{itemize}

\item K3 surfaces, $c_{2}=24$;

\item Enriques surfaces, $c_{2}=12$;

\item Abelian surfaces, $c_{2}=0$;

\item bielliptic surfaces, $c_{2}=0$;

\end{itemize}
for details see \cite[Ch.\ VI]{Bar}. Thus, let $(X, D)$ be as in Proposition \ref{reduction}. Since $\overline{c}_{2}=c_{2}=1$, we have that $X$ is the blowup an Abelian or bielliptic surface at exactly one point.

Now, let $D_{1}, ..., D_{k}$ be the irreducible components of the compactifying divisor $D$. Since each $D_{i}$ is a smooth elliptic curve with negative self-intersection, adjunction implies that $K_X \cdot D = - D_i^2$, so we have
\[
\overline{c}^{2}_{1}(X)= \big( K_{X}+\sum_{i}D_{i} \big)^{2}=K^{2}_{X}-\sum_{i}D^{2}_{i}=-1-\sum_{i}D^{2}_{i}.
\]
Then $3\overline{c}_{2}(X)=\overline{c}^{2}_{1}(X)$, which
implies
\begin{align}\notag
-D^{2}_{1}- \cdots -D^{2}_{k}=4.
\end{align}
Therefore, we have the following finite list of configurations:
\begin{itemize}

\item 1 cusp, $D^{2}_{1}=-4$;

\item 2 cusps, $D^{2}_{1}=-1$, $D^{2}_{2}=-3$ or $D^{2}_{1}=-2$, $D^{2}_{2}=-2$;

\item 3 cusps, $D^{2}_{1}=-1$, $D^{2}_{2}=-1$, $D^{2}_{3}=-2$;

\item 4 cusps, $D^{2}_{1}=D^{2}_{2}=D^{2}_{3}=D^{2}_{4}=-1$.

\end{itemize}

Let $(Y, C)$ denote the blow down configuration of $(X, D)$. Since $Y$ is an Abelian or bielliptic surface, $K_{Y}=0$. Thus, if $C_{i}$ is an irreducible component of $C$ in $Y$, adjunction implies that the arithmetic genus $p_a$ of $C_i$ satisfies
\[
p_{a}(C_{i})=1+\frac{C^{2}_{i}}{2},
\]
{see for example \cite[p.\ 13]{Friedman}.} Note that $C^{2}_{i}\geq-2$ and it is even. If {$C_i^{2}=-2$}, then $C_{i}$ is a smooth rational curve, which is impossible since $Y$ has universal cover $\cc^{2}$. If $C^{2}_{i}=0$, with $C_{i}$ nonsmooth, then {the genus-degree formula implies that} $C_{i}$ is a rational curve with a single node or a cusp. This is again impossible, since in both of these cases $C_{i}$ is normalized by $\pp^{1}$. In conclusion, either $C_{i}$ is a smooth elliptic curve with trivial self-intersection, or $C_{i}$ has a singular point $p$ and $C^{2}_{i}=2n$ for some $n\geq1$.

We study the singular case first. Let
\[
\pi:X\longrightarrow Y
\]
be the blowup map at $p$. We then have
\[
\pi^{*}C_{i}=D_{i}+rE
\]
where $D_{i}$ is the proper transform of $C_{i}$ in $X$, $E$ is the exceptional divisor, and $r$ is the multiplicity of the singular point $p$. Moreover, we have $D_{i}\cdot E=r$, $D^{2}_{i}=C^{2}_{i}-r^{2}$ and
\begin{equation}\label{genus}
2p_{a}(D_{i})-2=2p_{a}(C_{i})-2-r(r-1).
\end{equation}
If we want $D^{2}_{i}\leq-1$ with $C_{i}$ not smooth, we must have
\[
D^{2}_{i}=2n-r^{2} \le -1.
\]

Since $D_{i}$ is a smooth elliptic curve, Equation \eqref{genus} simplifies to the quadratic equation
\[
r^{2}-r-2n=0,
\]
whose solutions are given by
\[
r_{1,2}=\frac{1\pm\sqrt{1+8n}}{2}.
\]
Since $r$ is a positive integer, we only need to consider the positive square root case in the above formula. Therefore, the self-intersection of $D_{i}$ is given by
\[
2n - \left( \frac{1+\sqrt{1+8n}}{2} \right)^{2},
\]
for $n\geq1$. This self-intersection is easily seen to be decreasing in $n$ and less than $-4$ for $n\geq7$. All the possibilities for $1\leq n\leq 6$ are then given by the following list:
\begin{align}\label{list}
& n=1, \quad C^{2}_{i}=2, \quad r=2; \notag\\ & n=3, \quad
C^{2}_{i}=6, \quad r=3;\\ \notag & n=6, \quad C^{2}_{i}=12,\quad
r=4. \notag
\end{align}

In conclusion, we must understand whether or not an Abelian or bielliptic surface can support a curve with only one singular point of order $r$ and self-intersection as in \eqref{list}.

\subsection{The Abelian Case and the First Example}

We now study the case when $Y$ is an Abelian surface in detail. First, observe that the line bundle associated with a curve as in \eqref{list} must be ample.

\begin{lemma}\label{Nakai}
Let $C$ be an irreducible divisor on an Abelian surface $Y$ such that $C^{2}>0$. Then $L=\mathcal{O}_{Y}(C)$ is ample.
\end{lemma}

\begin{proof}
For any curve $E$ on $Y$, we would like to show that $C\cdot E>0$. Since $C^{2}>0$, we need to study curves $E\neq C$. For these curves we clearly have $C\cdot E\geq0$. Assume then that $C\cdot E=0$. Let $t_{y}(E)$ denote the translate of $E$ by an element $y\in Y$. Choosing $y\in Y$ appropriately, we can assume that $t_{y}(E)\cap C\neq\{0\}$. Since the curve $t_{y}(E)$ is numerically equivalent to $E$, we obtain a contradiction. We therefore conclude that $L$ is a strictly nef line bundle with positive self-intersection. The lemma is now a consequence of Nakai's criterion for ampleness of divisors on surfaces \cite[Cor.\ 6.4, p.\ 161]{Bar}.
\end{proof}

Next, we show that curves as in \eqref{list} cannot exist on an Abelian surface. The proof of this fact uses standard properties of theta functions. Recall that any effective divisor on a complex torus is the divisor of a theta function \cite[Th. 3.1]{Deb}. Let $C$ be a reduced divisor as in Lemma \ref{Nakai}. Then, if we let $V=\cc^{2}$ and $\pi: V\longrightarrow V/\Gamma$ be the universal covering map, we have that
\begin{equation}\label{theta}
\pi^{*}C=(\theta)
\end{equation}
for some theta function $\theta$ on $V$.

More precisely, we can find a Hermitian form $H$, a character $\alpha: \Gamma\longrightarrow U(1)$, and a theta function satisfying \eqref{theta} and the following ``normalized'' functional equation
\begin{equation}\label{functional}
\theta(z+\gamma)=\alpha(\gamma)e^{\pi H(\gamma, z)+\frac{\pi}{2}H(\gamma, \gamma)}\theta(z)=e_{\gamma}(z)\theta(z)
\end{equation}
for any $z\in V$ and $\gamma\in \Gamma$. Note that $e_{\gamma}$ is the factor of holomorphy for the line bundle $L=\mathcal{O}_{Y}(C)$, and we use the convention that the first variable of $H$ is the antiholomorphic variable. Then there is an identification between the space of sections of $L$ and the vector space of theta functions of type $(H, \alpha)$ on $V$.

Considering the list obtained in \eqref{list}, we are interested in the case when $C$ has exactly one singular point. Thus, let $C\in|L|$ be a reduced divisor and denote by $C^{*}=C\backslash\{p\}$ the smooth part of $C$. For every $q\in C^{*}$, $T_{q}C$ is a well defined $1$-dimensional subspace of $T_{q}Y$. Therefore, if we let $z_{1}, z_{2}$ be coordinate functions for $V$, the equation for $T_{q}C$ is given by
\[
\sum^{2}_{i=1}\partial_{z_{i}}\theta(q)(z_{i}-q_{i})=0.
\]
We can then consider a Gauss type map
\[
G: C^{*}\longrightarrow \pp^1
\]
where
\[
G(q)=(\partial_{z_{1}}\theta(q): \partial_{z_{2}}\theta(q)).
\]

We claim that since $C$ is reduced and $L$ is ample, the Gauss map cannot be constant. We proceed by contradiction. Suppose that the image of the Gauss map is the point $[x_{1}: x_{2}]\in \pp^1$. If $x_{2}\neq 0$, define the derivation
\[
\partial_{w}:=\partial_{z_{1}}-k\partial_{z_{2}},
\]
where $k=x_{1}/x_{2}$. If $x_{2}=0$, simply consider the derivative along the second coordinate function, in other words, $\partial_{w}=\partial_{z_{2}}$. By construction, we have $\partial_{w}\theta=0$ for all $q\in C^{*}$. Since $C$ is reduced, the function
\[
f=\partial_{w}\theta/\theta
\]
is holomorphic on $V$ except at the singular points of $\pi^{*}C$. By the Hartogs extension theorem, we know that $f$ can be extended to a holomorphic function on $V$. First, notice that
\begin{align*}
\partial_w \theta(z + \gamma) &= \partial_w e_\gamma(z) \theta(z) + e_\gamma(z) \partial_w \theta(z) \\
 &= \pi H(\gamma, v) e_\gamma(z) \theta(z) + e_\gamma(z) \partial_w \theta(z),
\end{align*}
where
\[
v=\partial_{w}\left(
\begin{array}{ccc}
z_1 \\
z_2
\end{array}\right).
\]
Therefore, the functional equation \eqref{functional} implies that
\[
f(z+\gamma)-f(z)=\pi H(\gamma, v)
\]
for any $\gamma\in\Gamma$, which further implies that
\[
f(z)=\pi H(z, v)+K
\]
for some constant $K$. Since $f$ is holomorphic and $H$ is antiholomorphic in $z$, we have therefore reached a contradiction. To summarize, we have shown that for any derivation $\partial_{w}$, the function $\partial_{w}\theta$ cannot be identically zero on $C^{*}$.

Now, by the functional equation given in \eqref{theta}, the restriction of $\partial_{w}\theta$ to $\pi^{*}C$ can be considered as a section of the line bundle $L$ restricted to $C$. Thus, the intersection number $(\partial_{w}\theta)\cdot C$ coincides with the self-intersection $C^{2}$. Now consider a derivation $\partial_{w}$ with parameter $w$ determined by a generic point in the image of the Gauss map, and suppose that the multiplicity of the singular point $p$ is $r_{p}$. The intersection number of $(\partial_{w}\theta)$ and $C$ at the singular point $p$ is then $r_{p}(r_{p}-1)$. Moreover, by construction $\partial_{w}\theta$ vanishes somewhere on $C^{*}$. We conclude that
\begin{equation}\label{obstruction}
C^{2}\geq r_{p}(r_{p}-1)+1.
\end{equation}

\begin{remark}\label{multiple}
The same argument shows that, if $C$ is an irreducible curve on an Abelian surface with $C^{2}>0$ and singular points $p_{j}$ of multiplicities $r_{j}$, then $C^{2}>\sum_{j}r_{j}(r_{j}-1)$.
\end{remark}

Next, we observe that in all of the cases given in \eqref{list}, we have
\[
C^{2}_{i}=r(r-1)
\]
so that, using \eqref{obstruction}, we can rule out the cases of one, two, and three cusps. Indeed, these are precisely the cases for which $C$ would contain an irreducible component that is forbidden by the above discussion. We summarize this as a lemma.

\begin{lemma}\label{Abelian}
Let $(X, D)$ be a toroidal compactification with $\overline{c}_{2}=1$ and $\kappa(X)=0$. If $X$ is not birational to a bielliptic surface, then $X$ is the blowup of an Abelian surface. Moreover, $D$ consists of four disjoint smooth elliptic curves with $D^{2}_{1}=D^{2}_{2}=D^{2}_{3}=D^{2}_{4}=-1$.
\end{lemma}

Thus, by Lemma \ref{Abelian}, we must classify the pairs $(Y, C)$ where $Y$ is an Abelian surface and $C$ consists of four smooth elliptic curves intersecting in just one point. We will show that, up to isomorphism, there is only one such pair. This result follows from few geometric facts.

\begin{fact}\label{fatto1}
Let $Y=\cc^{2}/\Gamma$ be an Abelian surface containing two smooth elliptic curves $C_{1}$, $C_{2}$ such that $C_{1}\cdot C_{2}=1$. Then $Y$ is isomorphic to the product $C_{1}\times C_{2}$.
\end{fact}
\begin{proof}
By translating the curves $C_{1}$ and $C_{2}$, we can always assume that $C_{1}\cap C_{2}= \{(0, 0)\}$. The curves $C_{i}$, $i=1, 2$, are then subgroups of $Y$. Thus, we can define the map
\[
\varphi: C_{1}\times C_{2}\longrightarrow Y
\]
that sends the point $(p, q)\in C_{1}\times C_{2}$ to $p-q\in Y$. The map $\varphi$ is clearly one-to-one.
\end{proof}

\begin{fact}\label{fatto2}
Let $Y=\cc^{2}/\Gamma$ be an Abelian surface containing three smooth elliptic curves $C_{i}$, $i=1, 2, 3$, such that $C_{1}\cap C_{2}\cap C_{3}=\{(0, 0)\}$ and such that $C_{i}\cdot C_{j}=1$ for any $i\neq j$. Then $Y$ is isomorphic to the product $C\times C$ where $C_{i}=C$ for any $i$.
\end{fact}
\begin{proof}
By Fact \ref{fatto1}, we have that $Y=C_{1}\times C_{2}$. Since $C_{3}\cdot C_{1}=1$, for $i=1, 2$, we conclude that $C_{3}=C_{i}$ for $i=1, 2$. The proof is complete.
\end{proof}

\begin{fact}\label{fatto3}
Let $Y$ be an Abelian surface that is the product of two identical elliptic curves, say $C=\cc/\Lambda$. Let $(w, z)$ be the natural product coordinates on $Y$. Then any smooth elliptic curve in $Y$, passing through the point $(0, 0)$, is given by an equation of the form $w=\alpha z$, with $\alpha$ such that $\alpha\Lambda\subseteq\Lambda$.
\end{fact}
\begin{proof}
A subgroup in $\cc^{2}$ is given by an equation of the form $w=\alpha z$. This equation descends to $Y$ precisely when $\alpha\Lambda\subseteq\Lambda$.
\end{proof}

\begin{fact}\label{fatto4}
Let $C_{\alpha}$ denote the curve in $Y=C\times C$ given by the equation $w=\alpha z$ with $\alpha\Lambda\subseteq\Lambda$ and $\alpha\neq 0$. Then $C_{0}\cdot C_{\alpha}=1$ if and only if $\alpha\Lambda=\Lambda$.
\end{fact}
\begin{proof}
The intersection $C_{0}\cap C_{\alpha}$ consists of $[\Lambda: \alpha\Lambda]$ distinct points, where $[\Lambda: \alpha\Lambda]$ denotes the index of the subgroup $\alpha\Lambda$ in $\Lambda$.
\end{proof}

Let us now return to our original problem of classifying all configurations of four elliptic curves $C_{i}$, $i=1, 2, 3, 4$, on an Abelian surface $Y$ such that
\[
C_{1}\cap C_{2}\cap C_{3}\cap C_{4}=\{p\}
\]
for a point $p\in Y$ and
\[
C_{i}\cdot C_{j}=1
\]
for any $i\neq j\in\{1, 2, 3, 4\}$. Any such configuration will be referred as \emph{good configuration}. By translating the $C_{i}$, we can assume the point $p$ coincides with the origin in $Y$. By Facts \ref{fatto1} and \ref{fatto2}, we can assume $Y=C\times C$ with the curves $C_{1}$ and $C_{2}$ being the factors in this splitting of $Y$. Then Facts \ref{fatto3} and \ref{fatto4} imply that we must look for values of $\alpha$, say $\alpha_{1}$ and $\alpha_{2}$, such that
\[
C_{3}=C_{\alpha_{1}}, \quad C_{4}=C_{\alpha_{2}}.
\]

For a generic elliptic curve $C=\cc/\Lambda$, the only values of $\alpha$ such that $\alpha\Lambda=\Lambda$ are given by $\alpha=\pm 1$. If this is the case, note that $C_{1}\cap C_{-1}$ consists of four {distinct} points. These points are exactly the $2$-torsion points of the lattice $\Lambda$. In conclusion, for a generic elliptic curve $C$, the Abelian surface $Y=C\times C$ cannot support a good configuration.

It remains to treat the case of a nongeneric elliptic curve $C$. Recall that there are only two elliptic curves with nongeneric automorphism group {\cite[\S IV.4]{Har}}. These elliptic curves correspond to the lattices $\zz[1, i]=\zz+\zz i$, $\zz[1, \zeta]=\zz+\zz\zeta$ where $\zeta=e^{\frac{\pi i}{3}}$.

For the lattice $\zz[1, i]$, we have four choices of the value of $\alpha$ so that $\alpha\zz[1, i]=\zz[1, i]$:
\[
\alpha=1, i, i^{2}, i^{3}.
\]
It turns out that none of the possible choices involving these parameters gives a good configuration. To see this, it suffices to observe that the configuration
\[
w=0,\quad z=0,\quad w=z, \quad w=iz,
\]
is such that
\[
C_{1}\cap C_{i}=\{(0, 0), \quad (1/2+i/2, 1/2+i/2)\}.
\]
{Notice that $C_1 \cap C_i$ is two points precisely because $1-i \in \zz[i]$ has norm $2$.} Any other configuration is either {equivalent} to the one above {by a self-isomorphism of $Y$}, or fails to be a good configuration by completely analogous reasons {in the sense that some pair of curves will intersect in at least two points because the difference between their slopes will not be a unit of $\zz[i]$}.

For the lattice $\zz[1, \zeta]$, we have six choices of the value of $\alpha$ so that $\alpha\zz[1, \zeta]=\zz[1, \zeta]$:
\[
\alpha=1, \zeta, \zeta^{2}, \zeta^{3}, \zeta^{4}, \zeta^{5}.
\]
Observe that
\[
w=0,\quad z=0,\quad w=z, \quad w=\zeta z,
\]
is a good configuration. In fact, the curves $C_{1}$ and $C_{\zeta}$ intersect at the points whose $z$-values satisfy
\begin{equation}\label{equality}
(\zeta-1)z=0\mod{\zz[1, \zeta]}.
\end{equation}
Since $(\zeta-1)=\zeta^{2}$, we conclude that
\[
C_{1}\cap C_{\zeta}=\{(0, 0)\}.
\]

We claim that this is the only good configuration {(note this is implicit in Hirzebruch's work \cite{Hir84})}. First, consider the configuration given by
\[
w=0,\quad z=0,\quad w=z, \quad w=\zeta^{2} z.
\]
Observe that the curves $C_{1}$ and $C_{\zeta^{2}}$, not only meet at the origin, but also in other two distinct points. These points are the two distinct zeros of the Weierstrass $\wp$-function associated with the lattice $\zz[1, \zeta]$. More precisely, we have
\[
C_{1}\cap C_{\zeta^{2}}=\{(0, 0),\quad ((1-\zeta^{2})/3, (1-\zeta^{2})/3),\quad ((\zeta^{2}-1)/3, (\zeta^{2}-1)/3)\}.
\]
As the reader can easily verify {by finding an explicit self-isomorphism of the abelian surface}, any other configuration can be reduced to the above two or to the configuration
\[
w=0,\quad z=0,\quad w=z, \quad w=-z,
\]
which we already know not to be good.

In conclusion, we have the following result, which proves the uniqueness of Hirzebruch's example among Abelian surfaces that are a smooth toroidal compactification of a ball quotient of Euler number one.

\begin{theorem}\label{final}
Let $(X, D)$ be a toroidal compactification with $\overline{c}_{2}=1$ and $\kappa(X)=0$ for which $X$ is not birational to a bielliptic surface. Then $X$ is the blowup of an Abelian surface $Y=\cc^{2}/\Gamma$ with $\Gamma=\zz[1, \zeta]\times\zz[1, \zeta]$ and $\zeta=e^{\frac{\pi i}{3}}$. Moreover, {up to a self-isomorphism of $Y$} the blowdown divisor $C$ of $D$ is given by
\[
w=0,\quad z=0,\quad w=z, \quad w=\zeta z,
\]
where $(w, z)$ are the natural product coordinates on $Y$. In other words, $(X, D)$ is the toroidal compactification with $\overline{c}_2 = 1$ described by Hirzebruch.
\end{theorem}

\section{The Bielliptic Case and Four More Examples}\label{Bagnera}

Recall that a \emph{bielliptic} surface is a minimal surface of Kodaira dimension zero and irregularity one. As shown at the beginning of the last century by Bagnera and de Franchis \cite{Bd07}, all such surfaces are finite quotients of products of elliptic curves. More precisely, we have the following classification theorem. For a modern treatment we refer to \cite[Ch.\ VI]{Bea}. 

\begin{theorem}[Bagnera--de Franchis, 1907]\label{bagnera}
{Let $X$ be a bielliptic surface. Then there are elliptic curves $E_{\lambda}$ and $E_{\tau}$ associated with the lattices $\zz[1, \lambda]$ and $\zz[1, \tau]$ such that $X$ is biholomorphic to $(E_{\lambda}\times E_{\tau})/G$, where G is a group of translations of $E_{\tau}$ that acts on $E_{\lambda}$ with $E_{\lambda}/G=\pp^{1}$. Moreover,} $G$ has one of the following types:
\begin{enumerate}

\item $G=\zz / 2 \zz$ acting on $E_{\lambda}$ by $x\rightarrow -x$;

\item $G=\zz / 2 \zz\times\zz / 2 \zz$ acting on $E_{\lambda}$ by
\[
x\rightarrow -x \quad\textrm{and}\quad x\rightarrow x+\alpha_{2},
\]
where $\alpha_{2}$ is a $2$-torsion point;

\item $G=\zz / 4 \zz$ acting on $E_\lambda$ by $x\rightarrow \lambda x$, where $\lambda = i$;

\item $G=\zz / 4 \zz\times \zz / 2 \zz$ acting on $E_\lambda$ by
\[
x\rightarrow \lambda x \quad\textrm{and}\quad x\rightarrow x+\frac{1+\lambda}{2},
\]
where $\lambda = i$;

\item $G=\zz / 3 \zz$ acting on $E_{\lambda}$ by $x\rightarrow \lambda x$, where $\lambda=e^{\frac{2\pi i}{3}}$;

\item $G=\zz / 3 \zz\times \zz / 3 \zz$ acting on $E_{\lambda}$ by
\[
x\rightarrow \lambda x \quad\textrm{and}\quad x\rightarrow x+\frac{1-\lambda}{3},
\]
where $\lambda=e^{\frac{2\pi i}{3}}$;

\item $G=\zz / 6 \zz$ acting on $E_{\lambda}$ by $x\rightarrow \zeta x$, where $\lambda=e^{\frac{2\pi i}{3}}$ and $\zeta=e^{\frac{\pi i}{3}}$.

\end{enumerate}
\end{theorem}

{Note that the action of $G$ on $Y = E_\lambda \times E_\tau$ is clearly free, since $G$ acts on $E_\tau$ by translations, so the map $Y \to X$ is an \'etale cover.} We now address the existence and uniqueness of toroidal compactifications of Euler number one that are birational to a bielliptic surface. First, observe that if such examples exist then they must be the blowup at just one point of a bielliptic surface by arguments in \S \ref{zero}. Second, we have the following.

\begin{lemma}\label{oneortwocusps}
Suppose that $X$ is the blowup of a bielliptic surface at exactly one point, and that $(X, D)$ is a smooth toroidal compactification of a complex hyperbolic manifold $M = X \smallsetminus D$. Then $M$ has either one or two cusps.
\end{lemma}
\begin{proof}
From \S \ref{zero}, $M$ has between one and four cusps. Since the Picard number of a bielliptic surface is two, the Picard number of the blowup is then three. Since the rank of the Picard group of $X$ is three, by {Theorem \ref{thm:CuspCount}} we have that $X$ can at most have two cusps.
\end{proof}

The problem is then reduced to the study of the existence of certain singular elliptic curves $C_i$ as in {Equation} \eqref{list} on a bielliptic surface. Regarding these possibilities, let us observe the following, which will allow us to make further reductions.

\begin{lemma}\label{fatto}
Let $Y$ be bielliptic and $\pi: E_{\lambda}\times E_{\tau}\rightarrow Y$ be the associated \'etale cover as in Theorem \ref{bagnera}. Suppose that $C_i$ is a curve from {Equation} \eqref{list} with a unique regular singular point of order $r \ge 2$. Then $\pi^{-1}(C_i)$ is the union of distinct smooth elliptic curves, and exactly $r$ of them pass through each lift of the singular point of $C_i$. In fact, $\pi^{-1}(C_i)$ contains exactly $r$ distinct irreducible components, and the stabilizer $G_{C_i'}$ of any irreducible component of $\pi^{-1}(C_i)$ has order $d/r$, where $d$ is the degree of $\pi$. In particular, $r$ divides $d$.
\end{lemma}

\begin{proof}
Let $d = |G|$ be the degree of $\pi$. Then there are exactly $d$ points on $E_{\lambda}\times E_{\tau}$ above the singular point $p$ of $C_i$. Consider an irreducible component $C_i'$ of $\pi^{-1}(C_i)$. Then Remark \ref{multiple} implies that $C_i'$ cannot be singular. Since $C_i$ is normalized by an elliptic curve, we see that $C_i'$ must be a smooth elliptic curve, hence $C_i'$ is an \'etale cover of the normalization of $C_i$.

Smoothness of each $C_i'$ implies there are exactly $r$ irreducible components of $\pi^{-1}(C_i)$ passing through any given point in $\pi^{-1}(p)$. To obtain a singularity of order exactly $r$ in the quotient, we see that the $G$-action on $E_{\lambda}\times E_{\tau}$ must identify exactly $r$ distinct smooth curves in each connected component of $\pi^{-1}(C_i)$. Thus, to prove the lemma, it suffices to show that $\pi^{-1}(C_i)$ is connected.

Since $C_i$ is an irreducible curve on $Y$ of positive self-intersection, it follows that $\pi^{-1}(C_i)$ has positive self-intersection. In fact, our {assumption that $C_i$ has positive self-intersection, along with the fact that $\pi$ is \'etale, implies} that $\pi^{-1}(C_i)$ is reduced and intersects each irreducible component $C_i'$ positively. {Then one sees exactly as in the proof of Lemma \ref{Nakai} that Nakai's criterion implies that $\pi^{-1}(C_i)$ is a reduced ample divisor.} It follows that the support of $\pi^{-1}(C_i)$ must be topologically connected {by Zariski's main theorem (e.g., applying \cite[Cor.\ III.11.3]{Har} to the linear system determined by the ample divisor)}, which proves the lemma.
\end{proof}

Now, we use the structure of the group $\mathrm{Num}(Y)$ {of divisors modulo numerical equivalence} to restrict the possible number of cusps even further. Let $Y$ be a bielliptic surface with associated group
\[
G = \zz / s \zz \times \zz / t \zz
\]
with $s \ge 2$ and $t \ge 1$. Set $d = |G| = s t$. Recall from \cite{Ser90} that $\mathrm{Num}(Y)$ has a $\zz$-basis consisting of $\frac{1}{s} A$ and $\frac{1}{t} B$, where $A$ and $B$ are the general fibers of the two fibrations of $Y$ associated with the coordinate projections of the Abelian variety from Theorem \ref{bagnera}. {Note that one projection is onto $E_{\lambda}/G=\pp^{1}$. Projection onto the elliptic curve $E_\tau/G$ is the Albanese fibration of $Y$ (see \cite[Ch.\ V, VIII]{Bea})}.

\begin{lemma}\label{first2cusp}
The case $C_1^{2} = C_2^2 = 2$ associated with $D_1^2 = D_2^2 = -2$ on a two-cusped manifold cannot occur.
\end{lemma}
\begin{proof}
From the list given in {Equation} \eqref{list}, we must have two curves $C_{1}$ and $C_{2}$ with self-intersection two, each with a singular point of {order} two. Since $C_1$ and $C_2$ only intersect at the singular point, which has {order} two on each, $C_{1}\cdot C_{2}=4$. Up to numerical equivalence, any curve $C_i$ in $Y$ can be written as $C_i=\frac{k_{1}}{s}A+\frac{k_{2}}{t}B$ with $k_{1}, k_{2}\in \zz$, and then
\begin{equation}\label{biellipticselfintersection}
C_i^2 = 2 \frac{k_1 k_2}{s t} A \cdot B = 2 k_1 k_2.
\end{equation}
For $C_i^2 = 2$, the only possibility is then $k_{1}=k_{2}= \pm 1$, the negative case being excluded by $C_i \cdot A \ge 0$. Therefore, we must have
\[
C_{1}=C_{2}=\frac{1}{s}A+\frac{1}{t}B
\]
in $\mathrm{Num}(Y)$. Then $C_{1}\cdot C_{2}=2\neq4$, which is a contradiction.
\end{proof}

\begin{lemma}\label{second2cusp}
The case $C_1^{2}=0$, $C_2^2 = 6$ associated with $D_1^2 = -1$ and $D_2^2 = -3$ on a two-cusped manifold cannot occur unless $3$ divides $|G|$.
\end{lemma}
\begin{proof}
In this case the curve $C_2$ has a regular singular point of {order} three. By Lemma \ref{fatto}, $\pi^{-1}(C_2)$ is the union of exactly $3 k$ smooth irreducible elliptic curves for some $k \ge 1$. Since $G$ must act transitively on these curves, the lemma follows.
\end{proof}

\begin{lemma}\label{first1cusp}
The case $C_1^{2}=12$ associated with $D_1^2 = -4$ on a one-cusped manifold cannot occur unless $|G|$ is divisible by $4$.
\end{lemma}
\begin{proof}
In this case the curve $C_1$ has a regular singular point of {order} four. The lemma follows exactly as Lemma \ref{second2cusp}.
\end{proof}

This covers all three possible cases with one or two cusps. We now have two more lemmas that we prove useful.

\begin{lemma}\label{Breduction}
Let $Y$ be a bielliptic surface with associated group $G$, {and let $C_1$ be a curve with a singular point $p$ of order $r$. Suppose that $(X, D_1)$ is a smooth toroidal compactification, where $X$ is the blowup of $Y$ at $p$ and $D_1$ is the proper transform of $C_1$. Choose generators} $\frac{1}{s}A$ and $\frac{1}{t} B$ for $\mathrm{Num}(Y)$ as above. Given $\frac{k_1}{s} A + \frac{k_2}{t} B \in Num(Y)$ representing the numerical class of a curve $C_1$ with a singular point $p$ of order $r$, we must have $k_1 > \frac{r s}{|G|}$ { and $k_2 > \frac{r t}{|G|}$}.
\end{lemma}
\begin{proof}
We have $A \cdot B = |G|$. Recall that $B$ is numerically equivalent to a general fiber of the Albanese map. Then, considering a fiber of the Albanese map through $p$,
\[
C_1 \cdot B = \frac{k_1 |G|}{s} \ge r,
\]
where $r$ is the order of the singular point.

Now consider the case of equality. There must be a smooth fiber $B_0$ (i.e., not a multiple fiber) of the Albanese map passing through $p$, and this fiber will intersect $C_1$ transversally with intersection number $r$, and $B_0$ is disjoint from $C_1$ away from $p$. Now, suppose that $(X, D_1)$ is a smooth toroidal compactification, where $X$ is the blowup of $Y$ at $p$ and $D_1$ is the proper transform of $C_1$. Then the proper transform $E_0$ of $B_0$ to $X$ is an elliptic curve disjoint from $C_1$ in $X$. In particular, it defines an elliptic curve on the complex hyperbolic manifold $X \smallsetminus D$, which is a contradiction. {The analogous argument for the fibration of the bielliptic surface over $\pp^1$ gives the bound on $k_{2}$.}
\end{proof}


Now we proceed to analyze the remaining possibilities.

\subsection{Initial reductions}\label{reductions1}

In this section, we show that the only bielliptic surfaces whose blowup can produce a smooth toroidal compactification of Euler number one are associated with the groups $\zz / 3 \zz$ and $\zz / 3 \zz \times \zz / 3 \zz$. The first case is immediate from the above.

\begin{corollary}\label{nozmod2}
No $\zz / 2 \zz$ bielliptic surface can produce a smooth toroidal compactification of Euler number one.
\end{corollary}
\begin{proof}
By Lemmas \ref{first2cusp} - \ref{first1cusp}, we see that $3$ or $4$ divides $|G|$, which is not true.
\end{proof}

We now proceed to rule out the remaining cases.

\begin{proposition}\label{nozmod4}
No blowup of a $\zz / 4 \zz$ bielliptic surface leads to a smooth toroidal compactification of Euler number one.
\end{proposition}
\begin{proof}
By Lemmas \ref{first2cusp} - \ref{first1cusp}, we must consider the one-cusped case. Here, the curve $C_1$ on $Y$ satisfies $C_1^{2}=12$, and has a regular singular point of {order} four. Write
\[
C_1=\frac{k_{1}}{4}A+k_{2}B
\]
as above. Then $C_1^2 = 2 k_1 k_2$, so we have four possibilities:
\begin{align*}
& k_{1}=1, \quad k_{2}=6, \quad C_1=\frac{1}{4}A+6B;\\ 
& k_{1}=2, \quad k_{2}=3, \quad C_1=\frac{1}{2}A+3B;\\ 
& k_{1}=3, \quad k_{2}=2,\quad C_1=\frac{3}{4}A+2B;\\ 
& k_{1}=6, \quad k_{2}=1,\quad C_1=\frac{3}{2}A+B.
\end{align*}

All four cases are impossible by applying Lemma \ref{Breduction}. {More precisely, we apply the bound on $k_1$ to the first three cases and the bound on $k_{2}$ to the last case.} This proves the proposition.
\end{proof}

\begin{proposition}
No blowup of a $\zz / 6 \zz$ bielliptic surface defines a smooth toroidal compactification of Euler number one.
\end{proposition}
\begin{proof}
In this case, we must consider the case where there are two curves $C_1, C_2$ on $Y$ such that $C_1$ is a smooth elliptic curve with $C_1^2 = 0$, $C_2^{2}=6$, and $C_2$ has a regular singular point of {order} three. Write
\[
C_2=\frac{k_{1}}{6}A+k_{2}B
\]
with the above notation. Then, $C_2^2 = 2 k_1 k_2$, so $C_2^2 = 6$ leaves us with the following possibilities:
\begin{align*}
& k_{1}=1, \quad k_{2}=3, \quad C_2=\frac{A}{6}+3B;\\ 
& k_{1}=3, \quad k_{2}=1, \quad C_2=\frac{A}{2}+B.
\end{align*}
The first case is impossible by Lemma \ref{Breduction}.\\

In the second case $C_2\cdot B=3$. Thus, assume we can find a curve $C_2$ with a singular point of order three in the numerical class of $\frac{A}{2}+B$. Let $\pi: E_{\rho}\times E_{\tau}\rightarrow Y$ be the covering of degree six. Moreover, we can assume that the $\zz / 6 \zz$ action is given by the group generated by
\[
\varphi(w, z)=\left(\zeta w, z+\frac{\gamma}{6}\right)
\]
for some $\gamma\in\zz[1, \tau]$ {and $\zeta = e^{\frac{\pi i}{3}}$}.

Using Lemma \ref{fatto}, $\pi^{-1}(C_2)$ must consist of a union of smooth elliptic curves. Moreover, $\pi^{-1}(C_2)$ consists of three smooth elliptic curves $E_{1}$, $E_{2}$, $E_{3}$ intersecting in {the six lifts of the points on $Y$ where $C_1$ and $C_2$ meet}. Thus, the automorphism group $\zz / 6 \zz$ must act transitively on these elliptic curves with isotropy group $\{1, \varphi^{3}\}$. Next, observe that, since $C_2\cdot B=3$, we have that $E_{i}\cdot F=1$ for any $i=1, 2, 3$ and any fiber $F$ of the map $E_{\rho}\times E_{\tau}\rightarrow E_{\tau}$.

Therefore, for any $i=1, 2, 3$ we can write $E_{i}=\{w=\alpha_{i}z+a_{i}\}$ for some appropriate complex numbers $\alpha_{i}$ and $a_{i}$. We then must have $\varphi^{3}E_{1}=E_{1}$ which then implies that
\[
\left(-\alpha_{1}z-a_{1}, z+\frac{\gamma}{2}\right)=\left(\alpha_{1}\left(z+\frac{\gamma}{2}\right)+a_{1}, z+\frac{\gamma}{2}\right).
\]
We then have $2\alpha_{1}z=-\alpha_{1} \frac{\gamma}{2}-2a_{1}$ modulo $\zz[1, \rho]$ for any $z\in E_{\tau}$, which is impossible unless $\alpha_{1}=a_{1}=0$. Reiterating this argument for the three elliptic curves we get that $\alpha_{i}=a_{i}=0$ for any $i=1, 2, 3$. This is impossible.
\end{proof}

\begin{proposition}\label{noz2timesz2}
No $\zz / 2 \zz\times\zz / 2 \zz$ bielliptic surface can determine a smooth toroidal compactification of Euler number one.
\end{proposition}
\begin{proof}
By Lemmas \ref{first2cusp} - \ref{first1cusp}, it suffices to consider the one cusp case, where
\[
C_1 = \frac{k_1}{2} A + \frac{k_2}{2} B,
\]
$C_1^2 = 12$, and $C_1$ has a regular singular point of {order} four. We have four possibilities:
\begin{align*}
& k_{1}=1, \quad k_{2}=6, \quad C_{1}=\frac{A}{2}+3B;\\ 
 & k_{1}=2, \quad k_{2}=3, \quad C_{1}=A+\frac{3B}{2};\\
 & k_{1}=3, \quad k_{2}=2, \quad C_{1}=\frac{3A}{2}+B;\\
 & k_{1}=6, \quad k_{2}=1, \quad C_{1}=3A+\frac{B}{2}.
\end{align*}
{The first three cases are eliminated by Lemma \ref{Breduction}.} For the fourth case we argue as follows. Note that in this case $C_{1}\cdot A=2$, which is a contradiction because we assumed that $C_{1}$ has a singular point of order four.
\end{proof}

\begin{proposition}\label{noz4z2}
No $\zz / 4 \zz\times\zz / 2 \zz$ bielliptic surface can determine a smooth toroidal compactification of Euler number one.
\end{proposition}
\begin{proof}
By Lemmas \ref{first2cusp} - \ref{first1cusp}, it suffices to consider the one cusp case, where $C_1^2 = 12$ and it has a regular singular point of {order} four.
If
\[
C_1 = \frac{k_{1}}{4}A+\frac{k_{2}}{2}B
\]
with $C_1^2 = 12$, we have four possibilities:
\begin{align*}
& k_{1}=1, \quad k_{2}=6, \quad C_{1}=\frac{A}{4}+3B;\\ 
 & k_{1}=2, \quad k_{2}=3, \quad C_{1}=\frac{A}{2}+\frac{3B}{2};\\
 & k_{1}=3, \quad k_{2}=2, \quad C_{1}=\frac{3A}{4}+B;\\
 & k_{1}=6, \quad k_{2}=1, \quad C_{1}=\frac{3A}{2}+\frac{B}{2}.
\end{align*}
The first two cases are impossible by Lemma \ref{Breduction}, {as is the fourth}.

For the third case, let us observe that $C_{1}\cdot B=6$ so that the curve $C_{1}$ is a $6$-section of the Albanese map $\pi_{2}: Y\rightarrow E_{\tau}/(\zz/4\zz\times\zz/2\zz)$. Let $\pi: E_{i}\times E_{\tau}\rightarrow Y$ be the degree eight \'etale cover and observe that $\pi^{-1}(B)=8E_{i}$ (here $i$ is a square root of $-1$). Next, let $H_{1}$, $H_{2}$, $H_{3}$, $H_{4}$ denote the irreducible components of $\pi^{-1}(C_1)$. Since the automorphism group $\zz/4\zz\times \zz/2\zz$ acts transitively on the $H_{j}$ and trivially on the numerical class of $E_{i}$, we obtain that for any $j=1, 2, 3, 4$ the $H_{j}$ is a $s$-section for a fixed integer $s$. This implies the contradiction $4s=6$.
\end{proof}

This leaves us with only $\zz / 3 \zz$ and $\zz / 3 \zz \times \zz / 3 \zz$. In fact, both will produce examples of smooth toroidal compactifications. We now proceed to analyze these cases and completely classify the examples they determine.

\subsection{The classification in the case of $\zz / 3 \zz$ quotients}\label{zeta3}

Let $Y$ be a $\zz / 3 \zz$ bielliptic quotient. By Theorem \ref{bagnera}, we can find two elliptic curves $E_{\rho}$ and $E_{\tau}$, associated with the lattices $\zz[1, \rho]$ and $\zz[1, \tau]$, respectively, such that
\[
Y=(E_{\rho}\times E_{\tau})/(\zz / 3 \zz).
\]
More precisely, $\rho=e^{\frac{2\pi i}{3}}$ while $\tau$ is arbitrary, and the $\zz / 3 \zz$ group of automorphisms of $E_{\rho}\times E_{\tau}$ is generated by the automorphism $\varphi(w, z)=(\rho w, z+\frac{\gamma}{3})$ for some $\gamma\in \zz[1, \tau]$.

Consider the group $\mathrm{Num}(Y)$ of divisors on $Y$ up to numerical equivalence. Given the bielliptic quotient $\pi:E_{\rho}\times E_{\tau}\rightarrow Y$, we have two elliptic fibrations
\begin{align*}
\pi_{1}: Y&\rightarrow E_{\rho}/(\zz / 3 \zz) = \pp^{1}\\
\pi_{2}: Y&\rightarrow E_{\tau}/(\zz / 3 \zz)
\end{align*}
with generic fibers $A$ and $B$. Recall from above that, up to numerical equivalence, we can write any curve $C \in \mathrm{Num}(Y)$ as
\[
C=\frac{k_{1}}{3}A+k_{2}B,
\]
for $k_{1}, k_{2} \in \zz$. Notice that all of the fibers of $\pi_{2}$ are {smooth and reduced} and that this map is none other than the Albanese map. The class $\frac{1}{3}A$ represents a multiple fiber of the map $\pi_{1}$ counted with multiplicity one. Moreover, we have $A\cdot B=3$.

By Lemmas \ref{first2cusp} - \ref{first1cusp}, the only possibility is a two-cusped manifold with $D_1^2 = -1$ and $D_2^2 = -3$, the latter of which determines a curve $C_2$ on $Y$ with $C_2^{2}=6$. Up to numerical equivalence, we have two possibilities:
\begin{align*}
& k_{1}=1, \quad k_{2}=3, \quad C_2=\frac{1}{3}A+3B;\\ 
 & k_{1}=3, \quad k_{2}=1, \quad C_2=A+B.
\end{align*}
The first case is ruled out by Lemma \ref{Breduction}.\\

It remains to discuss the case when the curve $C_2$ in $Y$ is numerically equivalent to $A+B$. Notice that in this case $C_2\cdot A=C_2\cdot B=3$. Let $E_{1}$, $E_{2}$, and $E_{3}$ denote the three smooth elliptic curves on $E_\rho \times E_\tau$ that are the irreducible components of $\pi^{-1}(C_2)$. Since $C_2\cdot A=3$, we have a one-to-one map from $E_i$ to $E_{\rho}$ for each $i=1, 2, 3$. Next, since $C_2\cdot B=3$ we also have a one-to-one map from each $E_i$ to $E_{\tau}$. Indeed, it follows that each $E_i$ has intersection number one with the general fiber of each factor projection. We therefore conclude that $E_\rho \cong E_\tau$ and $Y$ is a quotient of $E_{\rho}\times E_{\rho}$ by the group of automorphisms generated by the order three automorphism
\[
\varphi(w, z)=\left(\rho w, z+\frac{\gamma}{3}\right)
\]
for some $\gamma\in\zz[1, \rho]$.

The next step is to determine the admissible values for $\gamma$. First, observe that, up to a translation in the $w$-direction, we can assume that
\[
E_{1}=(\alpha_{1}z, z), \quad E_{2}=(\alpha_{2}z+a_{2}, z), \quad E_{3}=(\alpha_{3}z+a_{3}, z)
\]
where the $\alpha_{i}$'s and $a_{i}$'s are complex numbers to be determined. In particular, the $E_i$ cannot be three $\varphi$-translates of a curve with second coordinate zero, since they would project to a smooth curve on the quotient. Then, up to renumbering we have 
\[
\varphi (E_{1})=E_{2}, \quad \varphi (E_{2})=E_{3}, \quad \varphi (E_{3})=E_{1}.
\]
Analytically, this means that
\begin{align*}
\rho\alpha_{1}z &= \alpha_{2}\left(z+\frac{1}{3} \gamma \right)+a_{2}, \\
\rho\alpha_{2}z+\rho a_{2} &= \alpha_{3}\left(z+\frac{1}{3} \gamma \right)+a_{3}, \\
\rho\alpha_{3}z+\rho a_{3} &= \alpha_{1}\left(z+\frac{1}{3} \gamma \right),
\end{align*}
(recall that our coordinates are on the Abelian surface, so in $\cc^2$ our equations must be taken modulo $\zz[1, \rho]$) and we then obtain:
\begin{align*}
\rho\alpha_{1}=\alpha_{2}, \quad \rho\alpha_{2}&=\alpha_{3}, \quad \rho\alpha_{3}=\alpha_{1} \\
\frac{\alpha_2}{3} \gamma+a_{2}=0, \quad \rho a_{2}&=\frac{\alpha_3}{3} \gamma+a_{3}, \quad \rho a_{3}=\frac{\alpha_1}{3} \gamma.
\end{align*}
It follows immediately that $\rho^3 = 1$, i.e., that $\rho$ is a cube root of unity.

Since each $E_i$ is a $1$-section of the map $E_{\rho}\times E_{\rho}\rightarrow E_{\rho}$ that intersects $\{w=0\}$ in exactly one point, we have that $\{\alpha_{1}, \alpha_{2}, \alpha_{3}\}\in\{1, \zeta, ..., \zeta^{5}\}$, where $\zeta=e^{\frac{\pi i}{3}}$ (cf. Fact \ref{fatto4}). We also want these three sections to intersect in three distinct points, so $\{\alpha_{1}, \alpha_{2}, \alpha_{3}\}$ is of the form $\{u, u \rho, u \rho^2\}$ for some $6^{th}$ root of unity $u$ (not necessarily primitive).

Since these choices of slopes all differ by an automorphism of the Abelian surface, namely complex multiplication by $u$ on the first factor, it suffices to consider the first choice, $\alpha_{1}=1$, $\alpha_{2}=\rho$, $\alpha_{3}=\rho^{2}$. Note that complex multiplication also changes the $a_i$, but that is also forced by the above relationship between the $\alpha_i$ and $a_i$. Next, we want the automorphism group generated by $\varphi$ not only to act transitively on the curves $E_i$, but also on their intersection points
\[
E_{1}\cap E_{2}\cap E_{3}=\{ (z_{1}, z_{1}), (z_{2}, z_{2}), (z_{3}, z_{3})\}.
\]
This is necessary because we want the curve $C_2$ in $Y$ to have a unique singular point. Therefore, we have
\[
z_{2}-z_{1}=\frac{\gamma}{3}, \quad z_{3}-z_{2}=\frac{\gamma}{3}, \quad z_{3}-z_{1}=2\frac{\gamma}{3}.
\]
On the other hand, we have the identities
\[
z_{1}=\rho z_{1}+a_{2},\quad z_{2}=\rho z_{2}+a_{2},\quad z_{3}=\rho z_{3}+a_{2}
\]
which imply that 
\[
\rho (z_{2}-z_{1})=z_{2}-z_{1},\quad \rho (z_{3}-z_{1})=z_{3}-z_{1},\quad \rho (z_{3}-z_{2})=z_{3}-z_{2}.
\]
It follows that $\frac{\gamma}{3}=\frac{(1-\rho)}{3}$ or $\frac{\gamma}{3}=2\frac{(1-\rho)}{3}$, since $z_2 - z_1$ is a nonzero point on the elliptic curve $E_\rho$ stable under complex multiplication by $\rho$. We will later see that each choice gives an isomorphic quotient, so we assume for now that $\frac{\gamma}{3}=\frac{(1-\rho)}{3}$.

Using this information together with the previously derived formulas, we obtain 
\[
a_{2}=-\frac{(1-\rho)}{3}, \quad a_{3}=-2\frac{(1-\rho)}{3}.
\]
so that
\[
E_{1}=(z, z), \quad E_{2}= \left(\rho z-\frac{(1-\rho)}{3}, z \right), \quad E_{3}=\left(\rho^{2}z-2\frac{(1-\rho)}{3}, z\right).
\]
Next, we compute that 
\[
E_{1}\cap E_{2}\cap E_{3}=\Big\{\left(\frac{2}{3}, \frac{2}{3}\right), \quad \left(\frac{2\rho}{3}, \frac{2\rho}{3}\right), \quad \left(\frac{2\rho^{2}}{3}, \frac{2\rho^{2}}{3}\right)= \left(\frac{1+\rho}{3}, \frac{1+\rho}{3}\right)\Big\}.
\]
In conclusion, the curves $E_{1}$, $E_{2}$, $E_{3}$ are uniquely determined. Thus, the bielliptic surface $Y$ and the curve $C_2$ are also uniquely determined. We now show that $Y$ is independent of our choice of the $3$-torsion point.

\begin{lemma}\label{independence}
The bielliptic surface $Y$ and the curve $C_2$ described above are independent of the choice of $\frac{1}{3}\gamma$.
\end{lemma}

\begin{proof}
The two possibilities for the $\alpha_i$ and $\frac{1}{3} \gamma$ are:
\begin{itemize}

\item $(z, z)$, \quad\quad $\left(\rho z-\frac{(1-\rho)}{3}, z\right)$, \quad $\left(\rho^{2}z-2\frac{(1-\rho)}{3}, z\right)$

\item $(z, z)$, \quad\quad $\left(\rho z-2\frac{(1-\rho)}{3}, z\right)$, \quad $\left(\rho^{2}z-\frac{(1-\rho)}{3}, z\right)$

\end{itemize}
Let $X$ and $X'$ be the bielliptic surfaces obtained by taking the quotient of $E_{\rho}\times E_{\rho}$ by the automorphism groups generated by
\begin{align*}
\varphi(w, z)&=\left(\rho w, z+\frac{(1-\rho)}{3}\right) \\
\varphi'(w, z)&=\left(\rho w, z+2\frac{(1-\rho)}{3}\right),
\end{align*}
respectively. The first configuration must be considered in $X$, while the latter must be considered in $X'$. The self-isomorphism of $E_{\rho}\times E_{\rho}$ given by $\phi:(w, z)\rightarrow (-w, -z)$ takes the first configuration to the second and descends to an isomorphism between $X$ and $X'$. This proves that the two choices determine the same isomorphism class of bielliptic quotient $X$ with the same singular curve $C_2$.
\end{proof}

Next, we must find the possible elliptic curves $C_1$ in $Y$ intersecting $C_2$ in its unique singular point such that $C_1\cdot C_2=3$ (see the list given in \eqref{list}). The first and most obvious choice is to let $C_1$ be the unique fiber $B$ of the Albanese map passing through the singular point of $C_{2}$. Notice that since $B\cdot C_{2}=3$, the intersection is transverse. Thus, let $C_1 = B$ and consider the pair $(Y, C)$, where $C = B + C_2$, and let $X$ denote the blowup of $Y$ at the singular point $p$ of $C_2$. Let $D_i$ be the proper transform of $C_i$ and $D = D_1 + D_2$. We then claim that the pair $(X, D)$ is a smooth toroidal compactification. Indeed, it saturates the logarithmic Bogomolov--Miyaoka--Yau inequality:
\[
\overline{c}^{2}_{1}=(K_{X}+D)^{2}=K^{2}_{X}-D_1^{2}-D_2^{2}=-1+3+1=3\overline{c}_{2}.
\]

Next, we can take $C_1'$ as the unique fiber of $\pi_{2}$ passing through the singular point of $C_2$. Notice that this is not a multiple fiber and, since $C_2\cdot A=3$, the intersection is transverse. Consider the pair $(Y, C')$, where $C' = C_1' + C_2$, and blow up the singular point $p$ of $C_2$. Let $X$ denote the blowup of $Y$, $D_1', D_2$ denote the proper transforms of $C_1'$ and $C_2$, respectively, and set $D' = D_1' + D_2$. Again, it is easy to check that it saturates the logarithmic Bogomolov--Miyaoka--Yau inequality, so the pair $(X, D')$ is a smooth toroidal compactification.

Finally, we argue that these are the unique smooth toroidal compactifications coming from a $\zz / 3 \zz$ bielliptic surface. If $\frac{k_1}{3} A + k_2 B$ is another possible choice, it is a smooth elliptic curve of self-intersection zero, so $2 k_1 k_2 = 0$. In other words, it must be a multiple of $\frac{1}{3} A$ or $B$. This curve also must have intersection number $3$ with $C_2$, which leaves us only with the above two choices, $C_1$ and $C_1'$.

\subsection{Discussion of the Second and Third Examples}\label{secondthird}

Let $(X, D)$ and $(X, D')$ be the toroidal compactification found in \S \ref{zeta3}. We want to show that those compactifications are associated with two distinct complex hyperbolic surfaces. Assume this is not the case. There exists an automorphism $\Psi: X\rightarrow X$ sending $D_1' + D_2$ to $D_1 + D_2$, {since the map of complex hyperbolic surfaces takes cusps to cusps}. The first claim is that we necessarily must have $\Psi(D_2)=D_2$ and $\Psi(D_1')=D_1$. Observe that $X$ is the blowup at a single point of a bielliptic surface. Thus, inside $X$ there is a unique rational curve that is the exceptional divisor say $E$. This implies that $\Psi(E)=E$. Now $E\cdot D_2=3$ while $E\cdot D_1=E\cdot D_1'=1$. The claim then follows.

Let $Y$ denote the bielliptic surface obtained by contracting the exceptional divisor $E$. Also, let $C_1$, $C_1'$, and $C_2$ denote the blow down transform of the curves $D_1$, $D_1'$, and $C_2$. Next, let $\psi: Y\rightarrow Y$ denote the automorphism on $Y$ induced by $\Psi$ on $X$, {which exists since $\Psi$ must preserve the exceptional curve of the blowup}. Observe that $\psi(C_2)=C_2$ and $\psi(C_1')=C_1$.

\begin{fact}\label{fact}
There exists no such automorphism of $X$.
\end{fact}

\begin{proof}
Recall that $C_1$ is numerically equivalent to $B$ and $C_1'$ is numerically equivalent to $A$. Let $\psi_{*}: \mathrm{Num}(Y)\rightarrow \mathrm{Num}(Y)$ be the induced automorphism. Since $\psi(A)=B$, we have
\[
\psi_{*}\big( \frac{1}{3} A \big)= \frac{1}{3} B \notin \mathrm{Num}(Y),
\]
which is a contradiction.
\end{proof}

\subsection{The classification in the case of $\zz / 3 \zz\times \zz / 3 \zz$ quotients}\label{z3z3}

Let $Y$ be a $\zz / 3 \zz\times \zz / 3 \zz$ bielliptic quotient. By Theorem \ref{bagnera}, we can find two elliptic curves $E_{\rho}$ and $E_{\tau}$, respectively associated with the lattices $\zz[1, \rho]$ and $\zz[1, \tau]$, such that
\[
Y= (E_{\rho}\times E_{\tau})/(\zz / 3 \zz\times \zz / 3 \zz).
\]
More precisely, $\rho=e^{\frac{2\pi i}{3}}$ while $\tau$ is arbitrary, and the $\zz / 3 \zz\times \zz / 3 \zz$ group of automorphisms of $E_{\rho}\times E_{\tau}$ is generated by the commuting order three automorphisms 
\[
\varphi_{1}(w, z)=\left(\rho w, z+\frac{\gamma}{3}\right), \quad \varphi_{2}(w, z)=\left(w+ k \frac{(1-\rho)}{3}, z+\frac{\gamma'}{3}\right)
\]
for some $\gamma, \gamma' \in \zz[1, \tau]$ and $k = 1$ or $2$. Consider the group $\mathrm{Num}(Y)$ of divisors on $Y$ up to numerical equivalence. If $A$ and $B$ are the generic fibers of the two elliptic fibrations $\pi_{1}: Y\rightarrow \pp^{1}=E_{\rho}/(\zz / 3 \zz)$ and $\pi_{2}: Y\rightarrow E_{\tau}/(\zz / 3 \zz)$, then $\frac{1}{3}A$ and $\frac{1}{3}B$ form a basis of $\mathrm{Num}(Y)$. Therefore, up to numerical equivalence we can write any curve $C$ as $\frac{k_{1}}{3}A+\frac{k_{2}}{3}B$ for $k_{1}, k_{2} \in \zz$. Notice that $\pi_2$ is the Albanese map, and all its fibers are generic. The class $\frac{1}{3} A$ represents a multiple fiber of the map $\pi_{1}$ with multiplicity one. Finally, we have $A\cdot B=9$.

As in the $\zz / 3 \zz$ case, it suffices to consider the case where $C_1$ is an elliptic curve of self-intersection zero and $C_2^{2}=6$. Up to numerical equivalence, we can write $C_2=\frac{k_{1}}{3}A+\frac{k_{2}}{3}B$ where $\frac{1}{3}A$ and $\frac{1}{3}B$ are the above basis of $\mathrm{Num}(Y)$. There are two possibilities:
\begin{align*}
& k_{1}=1, \quad k_{2}=3, \quad C_2=\frac{1}{3}A+B;\\ 
 & k_{1}=3, \quad k_{2}=1, \quad C_2=A+\frac{1}{3}B.
\end{align*}

\begin{proposition}
The case $C_2=A+\frac{B}{3}$ cannot occur.
\end{proposition}
\begin{proof}
In this case $C_{2}\cdot A=3$ and $C_{2}\cdot B=9$. Then $\pi^*(C_2)$ is numerically equivalent to $3 E_\rho + 9 E_\tau$, and Lemma \ref{fatto} implies that it has three irreducible components, each of which is a smooth elliptic curve. Let $E_{1}$, $E_{2}$ and $E_{3}$ denote the three elliptic curves. We have a one-to-one map from $E_{i}$ to $E_{\rho}$ and a three-to-one map from $E_{i}$ to $E_{\tau}$ for each $i=1, 2, 3$.

There are two cases to consider, associated with the two isomorphism classes of degree $3$ quotients of $E_\rho$:
\[
E_\rho \times E_{\rho / (1-\rho)} \quad E_\rho \times E_{\rho/3},
\]
where $E_{\rho / (1-\rho)}$ is the quotient of $\cc$ by $\frac{1}{1-\rho}\zz[1, \rho]$ and $E_{\rho/3}$ denotes the quotient of $\cc$ by $\zz[1, \rho/3]$. In the first case, {notice} that $\zz[3, 1-\rho] = (1-\rho)\zz[1, \rho]$ is an index $3$ subring of $\zz[1, \rho]$, and $E_{\rho}\simeq E_{1-\rho} = \cc / \zz[3,1-\rho]$ under its self-isogeny of degree $3$. In order to simplify the computation we can replace $E_\rho$ with $E_{1-\rho}$, and the first case becomes $E_{1-\rho} \times E_\rho$. We first rule out that situation.

\begin{claim}
The case $E_{1 - \rho} \times E_\rho$ cannot occur.
\end{claim}
\begin{proof}
We work with coordinates $(w, z)$. Complex multiplication by $\rho$ on the curve $E_{1-\rho}$ has fixed points $0, 1, 2$. Therefore, the automorphisms of $E_{1-\rho}\times E_\rho$ of interest for bielliptic quotients are:
\[
\varphi_{1}(w, z)= \left(\rho w, z+\frac{\gamma}{3}\right), \quad \varphi_{2}(w, z)=\left(w+k, z+\frac{\gamma'}{3}\right)
\]
for some $\gamma, \gamma'\in\zz[1, \rho]$ and $k= 1$ or $2$.

Notice that no $E_i$ can be numerically equivalent to the second factor of the product, since the $E_i$ must {all} intersect in a point and the $(\zz / 3 \zz \times \zz / 3 \zz)$-orbit the second factor is a collection of disjoint curves. Since each $E_i$ is isomorphic to the first factor, up to translation in the $z$-direction we can assume
\[
E_{1}=(w, \alpha_{1}w), \quad E_{2}=(w, \alpha_{2}w+a_{2}), \quad \quad E_{3}=(w, \alpha_{3}w+a_{3})
\]
where the $\alpha_{i}$'s and $a_{i}$'s are, as of yet, unknown. To be well-defined, we only need $\alpha_i \in \frac{1}{1 - \rho} \zz[1, \rho]$, but since each $E_i$ has intersection number $3$ with the curve $(w, 0)$, we see that $\alpha_i$ is a unit of $\zz[1, \rho]$, i.e., each $\alpha_i$ is a power of $\zeta = e^{\pi i / 3}$.

We now claim that the $\zz/3\zz$ isotropy group of each $E_i$ is generated by $\varphi_{2}$. To prove this, by Lemma \ref{fatto}, the isotropy group of each $E_i$ in $\zz / 3 \zz \times \zz / 3 \zz$ is isomorphic to $\zz / 3 \zz$. Since $\zz / 3 \zz \times \zz / 3 \zz$ acts transitively on the $E_i$, the stabilizers of each $E_i$ are all the same subgroup. We must rule out $\varphi_1$, $\varphi_1 \varphi_2$, and $\varphi_1 \varphi_2^2$ from being in the isotropy group of $E_i$, and it suffices to focus on $E_1$.

If $\varphi_1(E_1) = E_1$, for every $w \in E_1$ we would have some $w' \in E_1$ such that
\[
\varphi_{1}(w, \alpha_{1}w)=\left(\rho w, \alpha_{1}w+\frac{\gamma}{3}\right)=(w', \alpha_{1}w')
\]
which then implies that
\begin{align*}
\rho w - w' &\equiv 0 \ \mod{\zz[3, 1-\rho]} \\
\alpha_{1}(w-w') &\equiv \frac{\gamma}{3} \mod{\zz[1, \rho]}
\end{align*}
for all $w\in E_{1-\rho}$, and hence
\[
\alpha_{1}(\rho - 1)w \equiv \frac{\gamma}{3} \mod{\zz[1, \rho]}
\]
for all $w$. This is clearly impossible for transcendental $w$.

Similarly, one can show that the isotropy group cannot be generated by $\varphi_{1} \varphi_{2}$. In fact, the identity $\varphi_{1}(\varphi_{2}(E_{1}))=E_{1}$ gives
\[
\alpha_{1}(\rho-1)w=-\alpha_{1}k+\frac{\gamma}{3}+\frac{\gamma'}{3} \mod{\zz[1, \rho]}
\]
for all $w\in E_{1-\rho}$. This again cannot hold for all $w$. The same argument rules out $\varphi_1 \varphi_2^2$ after replacing $k$ with $2 k$ and $\frac{\gamma'}{3}$ with $2 \frac{\gamma'}{3}$. Thus the isotropy group of each $E_i$ must be generated by $\varphi_2$.

Consequently, $\varphi_1$ acts transitively on the $E_i$, so up to renumbering we can assume
\[
\varphi_{1}(E_{1})=E_{2}, \quad\varphi_{1}(E_{2})=E_{3}, \quad\varphi_{1}(E_{3})=E_{1}.
\]
From $\varphi_1(E_1) = E_2$, we see that for all $w \in \cc$, there is a $w' \in \cc$ such that
\[
\left(\rho w, \alpha_1 w + \frac{\gamma}{3} \right) = (w', \alpha_2 w' + a_2).
\]
In other words, $\rho w - w' \in \zz[3, 1-\rho]$ and
\[
\alpha_1 w + \frac{\gamma}{3} - \alpha_2 w' - a_2 \in \zz[1, \rho].
\]
Combining these two congruences gives
\[
(\alpha_1 - \alpha_2 \rho)w + \frac{\gamma}{3} - a_2 \in \zz[1, \rho].
\]
Since this holds for all $w$, we must have $(\alpha_{2}\rho-\alpha_{1}) \equiv 0 \mod{\zz[1, \rho]}$. It then follows that $\frac{\gamma}{3} - a_2 \in \zz[1, \rho]$. Analogous arguments show that
\[
\alpha_{3}\rho-\alpha_{2},\ 2\frac{\gamma}{3} - a_3,\ \alpha_{1}\rho-\alpha_{3} \in \zz[1, \rho].
\]

Enumerating all the possibilities for the $\alpha_i$, we see that there is always some $\alpha_i$ that is either $\pm 1$. First, assume $\alpha_i = 1$. From $\varphi_2(E_i) = E_i$, we then see that
\[
\varphi_{2}(w, w + a_i)= \left(w+k, w+a_i + \frac{\gamma'}{3} \right)
\]
(with $a_1 = 0$) for all $w \in \cc$, where the first coordinate is taken modulo $\zz[3, 1-\rho]$ and the second is modulo $\zz[1, \rho]$. If this equals $(w', w' + a_i)$, then $w \equiv w' \mod{\zz[1, \rho]}$ from $w + k \equiv w' \mod{\zz[3, 1-\rho]}$ and $k = 1$ or $2$. Combining this with
\[
w + a_i + \frac{\gamma'}{3} - w' - a_i \in \zz[1, \rho]
\]
allows one to conclude that $\frac{\gamma'}{3} \in \zz[1, \rho]$, which is a contradiction. Taking $\alpha_i=-1$ leads to the exact same contradiction. This rules out the case $E_{1 - \rho} \times E_\rho$.
\end{proof}

\begin{claim}
The case $E_\rho \times E_{\rho / 3}$ also cannot occur.
\end{claim}
\begin{proof}
Consider the automorphisms
\[
\varphi_{1}(w, z)= \left(\rho w, z+\frac{\gamma}{3} \right), \quad \varphi_{2}(w, z)= \left(w+k \frac{(1-\rho)}{3}, z+\frac{\gamma'}{3} \right),
\]
that define our bielliptic quotient, where $\gamma, \gamma'\in \zz[1, \rho/3]$ and $k = 1$ or $2$. Since $E_i$ is isomorphic to the first factor of the product, up to translation in the $z$-direction we can write
\[
E_{1}=(w, \alpha_{1}w), \quad E_{2}=(w, \alpha_{2}w+a_{2}), \quad \quad E_{3}=(w, \alpha_{3}w+a_{3}),
\]
where the $\alpha_{i}$'s and $a_{i}$'s are complex numbers to be determined. First, to give a well-defined elliptic curve, $\alpha_i$ must have the property that $\alpha_i \lambda \in \zz[1, \frac{\rho}{3}]$ for all $\lambda \in \zz[1, \rho]$, which actually implies that $\alpha_i \in \zz[1, \rho]$. Next, since for each $i=1, 2, 3$ we want $E_{i}\cdot (w, 0)=3$ {to obtain the desired configuration on the bielliptic surface}, $\alpha_{i}$ must be a root of unity.

We now claim that the group generated by $\varphi_{1} \varphi_2^r$ cannot be the isotropy group of the $(\zz / 3 \zz \times \zz / 3 \zz)$-action for $r = 0,1,2$. If $\varphi_{1} \varphi_2^r (E_{1})=E_{1}$, then for all $w \in \cc$, there is a $w' \in \cc$ such that
\[
\left(\rho w + r k \frac{(1 - \rho)}{3}, \alpha_1 w + \frac{\gamma}{3} + r \frac{\gamma'}{3} \right) = (w', \alpha_1 w')
\]
on $E_\rho \times E_{\rho / 3}$. This implies that $\rho w - w'$ is congruent to $r k \frac{(1-\rho)}{3}$ modulo $\zz[1, \rho]$, and we can conclude that
\[
\alpha_1(1 - \rho)w + \frac{\gamma}{3} + r \frac{\gamma'}{3} \in \zz[1, \rho/3].
\]
This cannot hold for all $w \in \cc$, so $\varphi_1 \varphi_2^r$ cannot stabilize $E_1$.

Then, up to renumbering we can assume
\[
\varphi_{1}(E_{1})=E_{2}, \quad\varphi_{1}(E_{2})=E_{3}, \quad\varphi_{1}(E_{3})=E_{1}.
\]
This means that, as points on $E_\rho \times E_{\rho / 3}$, we must have
\[
(\alpha_{2}\rho-\alpha_{1})=0,\quad a_{2}=\frac{\gamma}{3},\quad (\alpha_{3}\rho-\alpha_{2})=0,\quad a_{3}=2\frac{\gamma}{3},\quad (\alpha_{1}\rho-\alpha_{3})=0.
\]
This implies that $\{\alpha_{1}, \alpha_{2}, \alpha_{3}\}$ is of the form $\{u, u \rho^{2}, u \rho\}$ for some (possibly not primitive) $6^{th}$ root of unity $u$. Therefore, up to the automorphism of $E_\rho \times E_{\rho / 3}$ generated by complex multiplication by $u$ on $E_\rho$, we can assume that $u = 1$.

Since $\varphi_{2}(E_{1})=E_{1}$ and $E_1$ is the curve $(w,w)$, we see that
\[
k \frac{1-\rho}{3}=\frac{\gamma'}{3} \mod{\zz[1, \rho/3]}.
\]
Since $\frac{1-\rho}{3}=\frac{1}{3}\mod{\zz[1, \rho/3]}$, we obtain $\gamma'=k$, which is consistent with $\varphi_{2}(E_{i})=E_{i}$ for all $i=1, 2, 3$.

In conclusion, we have the configuration
\[
E_{1}=(w, w), \quad E_{2}=\left(w, \rho^{2}w+\frac{\gamma}{3}\right), \quad E_{3}=\left(w, \rho w+2\frac{\gamma}{3}\right)
\]
for some $\gamma \in \zz[1, \rho/3]$. Recall that we must have
\[
E_1 \cap E_2 = E_1 \cap E_3 = E_2 \cap E_3.
\]
For a point $(w,w)$ on $E_1 \cap E_2$, we have
\[
(1 - \rho^2) w - \frac{\gamma}{3} \in \zz[1, \rho/3],
\]
and for a point $(w,w)$ on $E_1 \cap E_3$ we similarly have
\[
(1 - \rho) w - 2 \frac{\gamma}{3} \in \zz[1, \rho/3].
\]

Consider the point $w = \frac{\gamma}{3(1-\rho^2)}$ on $E_1 \cap E_2$. For this point to lie on $E_1 \cap E_3$, we have
\[
(1-\rho) \left(\frac{\gamma}{3(1-\rho^2)}\right) - 2 \frac{\gamma}{3} = -(\rho + 2)\frac{\gamma}{3} \in \zz[1,\rho/3].
\]
Write $\gamma = a + b \frac{\rho}{3}$ for $a,b \in \{0, 1, 2\}$, which we can do because we only care about $\frac{\gamma}{3}$ as a $3$-torsion point on $E_{\rho/3}$. Then
\[
(-\rho - 2) \frac{\gamma}{3} = -\frac{6 a - b}{9} - \frac{3a + b}{9} \rho \in \zz[1, \rho/3].
\]
However, then $6 a - b$ is divisible by $9$ for $a,b \in \{0,1,2\}$. This is only possible if $a = b = 0$, which implies that we can take $\gamma = 0$ in the definition of $\varphi_1$. However, this is impossible, since $\varphi_1$ then cannot act freely on $E_\rho \times E_{\rho / 3}$. This contradiction proves the claim.
\end{proof}
This rules out all possibilities where $C_2$ is numerically equivalent to $A + \frac{1}{3} B$, which proves the proposition.
\end{proof}

We are then left with the case $C_2=\frac{1}{3}A+B$. Let $E_{1}$, $E_{2}$, and $E_{3}$ denote the three smooth elliptic curves in $\pi^{-1}(C_2)$. Since $C_2\cdot A=9$, we have that for each $i=1, 2, 3$ there is a three-to-one map from $E_{i}$ to $E_{\rho}$. Next, since $C_2\cdot B=3$ we have a one-to-one map from $E_{i}$ to $E_{\tau}$ for each $i=1, 2, 3$. We therefore conclude that $Y$ is a quotient of $E_{\rho}\times E_{\tau}$, where $E_{\tau}$ is a degree three cover of $E_{\rho}$.

As in previous cases, up to translation in the $w$-direction we can write
\[
E_{1}=(\alpha_{1}z, z), \quad E_{2}=(\alpha_{2}z+a_{2}, z), \quad E_{3}=(\alpha_{3}z+a_{3}, z).
\]
Here, each $E_{i}$ is a $1$-section of the map $E_{\rho}\times E_{\tau}\rightarrow E_{\tau}$ and $E_{i}\cdot E_{\tau}=3$, and it follows that each $\alpha_i$ is a power of $\zeta = e^{\pi i / 3}$.

Up to isomorphism, the degree three covers of $E_{\rho}$ are associated with the lattices $\zz[3, 1-\rho]$ and $\zz[3, \rho]$, {as one can easily see by enumerating the index $3$ subgroups of $\zz[1, \rho]$}. For simplicity, we let $E_{\tau}$ denote one of these degree three covers of $E_{\rho}$. The $\zz / 3 \zz\times \zz / 3 \zz$ automorphism group of $E_{\rho}\times E_{\tau}$ is then generated by the following commuting automorphisms of order three:
\[
\varphi_{1}(w, z)=\left(\rho w, z+\frac{\gamma}{3}\right), \quad \varphi_{2}(w, z)=\left(w+ k \frac{(1-\rho)}{3}, z+\frac{\gamma'}{3}\right)
\]
where $\gamma, \gamma'\in\zz[3, \rho]$ or $\gamma, \gamma'\in\zz[3, 1-\rho]$ and $k$ is an integer.

Next, the $\zz / 3 \zz\times \zz / 3 \zz$ group of automorphism must act transitively on the three elliptic curves $E_{1}, E_{2}, E_{3}$ with isotropy group $\zz / 3 \zz$. We claim that the isotropy group must be generated by $\varphi_{2}$. If not, assume that $\varphi_{2}(E_{1})=E_i$ for some $i \neq 1$. This implies that for all $w \in \cc$, there is a $w' \in \cc$ such that
\begin{align*}
w' &\equiv w + \frac{\gamma'}{3} \mod{\zz[1, \rho]} \\
\alpha_i w' + a_i &\equiv \alpha_1 w + k \frac{(1-\rho)}{3} \quad \mod{\zz[1, \rho]}.
\end{align*}
Combining these gives
\[
(\alpha_1 - \alpha_i) w - \alpha_i \frac{\gamma'}{3} + k \frac{(1-\rho)}{3} - a_i \in \zz[1, \rho]
\]
for all $w \in \cc$. Taking a transcendental $w$, this is impossible unless $\alpha_i = \alpha_1$, which is a contradiction. Therefore $\varphi_2$ is contained in the isotropy group of $E_1$, and hence generates the isotropy group of every $E_i$.

Now, up to renumbering we can assume $\varphi_1(E_i) = E_{i + 1}$, where the index $i$ is considered modulo $3$. We then have that
\[
\rho\alpha_{1}-\alpha_{2},\quad \rho\alpha_{2}-\alpha_{3},\quad \rho\alpha_{3}-\alpha_{1} \in \zz[1, \rho].
\]
As in previous cases, this implies that $\{\alpha_{1}, \alpha_{2}, \alpha_{3}\}$ is of the form $\{u, u \rho, u \rho^{2}\}$ for some $6^{th}$ root of unity $u$, and up to an automorphism of the Abelian surface, we can assume $u = 1$.

Then, $\varphi_{2}(E_{1})=E_{1}$ implies that $\frac{\gamma'}{3} \equiv k \frac{1-\rho}{3}$ modulo $\zz[1, \rho]$. Assume that $\gamma'\in \zz[3, \rho]$. We then have $\gamma'=3a+b\rho$ with $a, b\in\zz$, which implies that $k$ is $3$. However, $k$ must be congruent to $1$ or $2$ modulo $3$, since the $(\zz/3\zz \times \zz/3\zz)$-action {by definition has} nontrivial translation part on the first factor, so this is a contradiction.

We therefore conclude that $E_{\tau} = E_{1-\rho}$ is associated with the lattice $\zz[3, 1-\rho]$ and that $\gamma'=k(1-\rho)$ for $k$ an integer, since
\[
k(1 - \rho) - \gamma' \in 3 \zz[1, \rho] \subset \zz[3, 1-\rho].
\]
From the fact that $\varphi_{1}(E_{1})=E_{2}$ we obtain that $a_{2}=-\rho \frac{\gamma}{3}$. Similarly, since $\varphi_{1}(E_{2})=E_{3}$, we have $a_{3}=-2 \rho \frac{\gamma}{3}$. In conclusion, we have that the curves $E_{i}$ are uniquely determined by the equations
\[
E_{1}=(z, z), \quad E_{2}=\left(\rho z - \rho \frac{\gamma}{3}, z\right), \quad E_{3}=\left(\rho^{2}z - 2 \rho \frac{\gamma}{3}, z \right).
\]

Suppose first that $\gamma = 3 {\in \zz[3, 1-\rho]}$. We then have
\[
E_{1}=(z, z), \quad E_{2}=(\rho z, z), \quad E_{3}=(\rho^{2}z, z).
\]
and see that $E_{1}\cap E_{2}=E_{1}\cap E_{3}=E_{2}\cap E_{3}$ and $E_{1} \cap E_{2}$ equals
\begin{align*}
\Big\{&(0, 0), (0, 1), (0, 2), \\
&\left(\frac{1-\rho}{3}, \frac{1-\rho}{3}\right), \left(\frac{1-\rho}{3}, 1+\frac{1-\rho}{3}\right), \left(\frac{1-\rho}{3}, 2+\frac{1-\rho}{3}\right), \\
&\left(2 \big(\frac{1-\rho}{3}\big), 2 \big(\frac{1-\rho}{3}\big)\right), \left(2 \big(\frac{1-\rho}{3}\big), 1+2 \big(\frac{1-\rho}{3}\big)\right), \left(2 \big(\frac{1-\rho}{3}\big), 2+2 \big(\frac{1-\rho}{3}\big) \right) \Big\}.
\end{align*}
The $\zz/3\zz\times\zz/3\zz$ action acts transitively on these nine points, so we obtain a curve with a unique singular point of {order} three in the bielliptic quotient. In other words, the choice $\gamma = 3$ determines a bielliptic surface $Y$ with a curve $C_2$ numerically equivalent to $\frac{1}{3}A+B$ with a unique singular point of {order} $3$.

\begin{proposition}\label{FINALLYDONE}
Any other choice of $\gamma \in \zz[3, 1-\rho]$ either gives an isomorphic configuration as $\gamma = 3$ or gives a configuration that descends to a curve in the quotient with three singular points (and hence cannot determine a smooth toroidal compactification of Euler number one).
\end{proposition}
\begin{proof}
Define automorphisms
\[
\psi_1(w, z) = (-w, -z), \quad \psi_2(w, z) = \left(w + \frac{2}{3}, z + \frac{2}{3} \right)
\]
of $E_\rho \times E_{1 - \rho}$. Suppose that $\frac{\gamma}{3}$ a $3$-torsion point on $E_{1 - \rho}$, and consider the curves
\[
E_1 = (z, z) \quad E_2(\gamma) = \left(\rho z + \rho \frac{\gamma}{3}\right) \quad E_3(\gamma) = \left(\rho^2 + 2 \rho \frac{\gamma}{3}\right).
\]
A direct calculation shows that:
\begin{align*}
\psi_i(E_1) &= E_1 \quad i = 1,2\\
\psi_1(E_2(\gamma)) &= E_2(-\gamma) \\
\psi_2(E_2(\gamma)) &= E_2(\gamma + 2(1 - \rho)) \\
\psi_2^2(E_2(\gamma)) &= E_2(\gamma + (1 - \rho)) \\
\psi_1(E_3(\gamma)) &= E_3(-\gamma) \\
\psi_2(E_3(\gamma)) &= E_3(\gamma + 2(1 - \rho)) \\
\psi_2^2(E_3(\gamma)) &= E_3(\gamma + (1 - \rho))
\end{align*}
These maps descend to isomorphisms of the associated bielliptic quotients, and the orbit of $\gamma = 3$ under this action has cardinality six. In particular, six of the eight nontrivial $3$-torsion points on $E_{1 - \rho}$ all determine the same pair $(Y, C_2)$.

The remaining points not in the orbit of $\gamma = 3$ under the group generated by $\psi_1$ and $\psi_2$ are $\gamma = (1-\rho)$ and $2(1-\rho)$. We now show that these cases cannot give rise to a smooth toroidal compactification of Euler number one.

For $\gamma=(1-\rho)$, we have
\[
E_{1}=(z, z), \quad E_{2}=\left(\rho z-\frac{(1-\rho)}{3}, z\right), \quad E_{3}=\left(\rho^{2}z-2\frac{(1-\rho)}{3}, z \right).
\]
and see that $E_{1}\cap E_{2}=E_{1}\cap E_{3}=E_{2}\cap E_{3}$ and $E_{1} \cap E_{2}$ equals
\begin{align*}
\Big\{&\left(\frac{2}{3}, \frac{2}{3}\right), \left(\frac{2}{3}, \frac{2}{3}+1\right), \left(\frac{2}{3}, \frac{2}{3}+2\right), \\
&\left(\frac{2}{3}+\frac{1-\rho}{3}, \frac{2}{3}+\frac{1-\rho}{3}\right), \left(\frac{2}{3}+\frac{1-\rho}{3}, \frac{2}{3}+1+\frac{1-\rho}{3}\right), \\
&\left(\frac{2}{3}+\frac{1-\rho}{3}, \frac{2}{3}+2+\frac{1-\rho}{3}\right), \left(\frac{2}{3}+2 \big(\frac{1-\rho}{3}\big), \frac{2}{3}+2 \big(\frac{1-\rho}{3}\big)\right), \\
&\left(\frac{2}{3}+2 \big(\frac{1-\rho}{3}\big), \frac{2}{3}+1+2 \big(\frac{1-\rho}{3}\big)\right), \left(\frac{2}{3}+2 \big(\frac{1-\rho}{3}\big), \frac{2}{3}+2+2 \big(\frac{1-\rho}{3}\big) \right) \Big\}.
\end{align*}
Next, observe that the orbit of the point $(\frac{2}{3}, \frac{2}{3})$ under the action of the order three isomorphism $\varphi_{1}$ is 
\[
\left\{(\frac{2}{3}, \frac{2}{3}), \left(\frac{2}{3}+\frac{1-\rho}{3}, \frac{2}{3}+\frac{1-\rho}{3}\right), \left(\frac{2}{3}+2 \big(\frac{1-\rho}{3}\big), \frac{2}{3}+2 \big(\frac{1-\rho}{3}\big)\right)\right\}.
\]
Now,
\[
\varphi_{2}(w, z)= \left(w+k\frac{(1-\rho)}{3}, z+k\frac{(1-\rho)}{3}\right),
\]
so that the orbit of of the point $(\frac{2}{3}, \frac{2}{3})$ under this automorphism is the same as above. Therefore, the curve $C$ determined by the images of the $E_{i}$'s in the bielliptic quotient determined by $\varphi_1$ and $\varphi_2$ has three {order} three singular points. In particular, it does not satisfy the criteria necessary to determine a smooth toroidal compactification of Euler number one. {Indeed, the proper transform of this curve under the blowup at one point will remain singular, and hence cannot be one of the smooth elliptic curves in a compactification divisor.}

The case $\gamma=2(1-\rho)$ is isomorphic to the first under $\psi_1$, and hence also cannot occur. This rules out $\gamma = k(1 - \rho)$, $k = 1,2$ from consideration, and completes the proof of the proposition.
\end{proof}

Thus the bielliptic surface $Y$ and the curve $C_2$ are therefore uniquely determined. Next, we have to find smooth elliptic curves $E$ in $Y$ intersecting $C_2$ in its unique singular point such that $C_2\cdot E=3$.

First, consider the unique fiber $C_1$ of the Albanese map, which is numerically equivalent to $B$, passing through the singular point of $C_2$. Since $C_2\cdot B=3$, the intersection is transverse. Thus, consider the pair $(Y, C)$, where $C = C_1 + C_2$, and let $X$ be the blowup of $Y$ at the singular point $p$ of $C_2$. Let $D_1$ and $D_2$ be the proper transforms of $D_1$ and $D_2$, respectively, and set $D = D_1 + D_2$. Then $(X, D)$ saturates the logarithmic Bogomolov--Miyaoka--Yau inequality, and hence defines a smooth toroidal compactification. 

Next, notice that the singular point $p\in C_2$ corresponds to the point $(0, 0)\in E_{\rho}\times E_{\tau}$. Therefore, the fiber $C_1'$ of $\pi_{2}$ passing through the point $p$ must be a multiple fiber, and $C_1'$ is numerically equivalent to $\frac{A}{3}$. Since $C_2\cdot \frac{A}{3}=3$, the intersection of $C_1'$ with $C_2$ is transverse. Consequently, consider the pair $(Y, C)$ with $C = C_1' + C_2$, let $X$ be the blowup of $Y$ at the singular point $p$ of $C_2$, and let $D_1', D_2$ be the proper transforms of $C_1'$ and $C_2$, respectively. Again, $(X, D')$, where $D' = C_1' + C_2$, saturates the logarithmic Bogomolov--Miyaoka--Yau inequality.

The exact same argument as in \S \ref{zeta3} shows that $(X, D)$ and $(X, D')$ are the unique smooth toroidal compactifications arising from a $\zz / 3 \zz \times \zz / 3 \zz$ bielliptic surface. Indeed, we again see that such a curve $C_1$ numerically equivalent to $\frac{k_1}{3} A + \frac{k_2}{3} B$ has self-intersection $2 k_1 k_2 = 0$, so $C_1$ is a multiple of $\frac{1}{3}A$ or $\frac{1}{3}B$. Then $C_1 \cdot C_2 = 3$ leaves us with only the two possibilities considered above.

\subsection{Discussion of the Fourth and Fifth Examples}\label{fourthfifth}

Let $(X, D)$ and $(X, D')$ be the toroidal compactifications found in \S \ref{z3z3}. We want to show that these compactifications are associated with two distinct complex hyperbolic surfaces. Assume this is not the case. There exists an automorphism $\Psi: X\rightarrow X$ sending $D_1' + D_2$ to $D_1 + D_2$. The first claim is that we necessarily must have $\Psi(D_2)=D_2$ and $\Psi(D_1')=D_1$. Observe that $X$ is the blowup at a single point of a bielliptic surface. Thus, inside $X$ there is a unique rational curve that is the exceptional divisor say $E$. This implies that $\Psi(E)=E$. Now $E\cdot D_2=3$ while $E\cdot D_1=E\cdot D_1'=1$. The claim then follows.

Let $Y$ denote the bielliptic surface obtained by contracting the exceptional divisor $E$. Also, let $C_1$, $C_1'$, and $C_2$ denote the blow down transform of the curves $D_1$, $D_1'$, and $C_2$. Next, let $\psi: Y\rightarrow Y$ denote the automorphism on $Y$ induced by $\Psi$ on $X$. Observe that $\psi(C_2)=C_2$ and $\psi(C_1')=C_1$.

\begin{fact}\label{fact1}
There exists no such automorphism of $X$.
\end{fact}

\begin{proof}
Recall that $C_1'$ is numerically equivalent to $\frac{A}{3}$ and $C_1$ is numerically equivalent to $B$. Let $\psi_{*}: \mathrm{Num}(Y)\rightarrow \mathrm{Num}(Y)$ be the induced automorphism. Since $\psi(C_{2})=C_{2}$ and $\psi(C_{1}')=C_{1}$, we have $\psi_{*}(B)=\frac{A}{3}$ and then
\[
\psi_{*}\big( \frac{1}{3} B \big)= \frac{1}{9} A \notin \mathrm{Num}(Y),
\]
which is a contradiction.
\end{proof}

\section{Recap of the Five Examples}\label{recap}

We now give a concise recap of the five examples constructed above. Recall that $\rho = e^{2 \pi i / 3}$, $\zeta = e^{\pi i / 3}$, and $E_\rho$ is the elliptic curve $\cc / \zz[1, \rho]$.

\subsection{Example 1}

Consider $E_\rho \times E_\rho$ with coordinates $(w, z)$, and consider the curves $C_1^{(1)}, \dots, C_4^{(1)}$ on $Y_1$ defined by
\[
w=0,\quad z=0,\quad w=z, \quad w=\zeta z.
\]
Then $C_1^{(1)} \cap \cdots \cap C_4^{(1)} = \{(0, 0)\}$. Let $X_1$ be the blowup of $E_\rho \times E_\rho$ at $(0,0)$, $D_i^{(1)}$ be the proper transform of $C_i^{(1)}$ to $X_1$, and $D_1 = \sum D_i^{(1)}$. Our first example (originally due to Hirzebruch \cite{Hir84}), is the pair $(X_1, D_1)$.

\subsection{Example 2}

Let $Y_2$ be bielliptic quotient of $ E_\rho \times E_\rho$ defined by the automorphism
\[
\varphi(w, z) = \left( \rho w, z + \frac{(1 - \rho)}{3} \right)
\]
{of order $3$}. Let $C_1^{(2)}$ be the image on $Y_2$ of the curve $w = 0$ on $E_\rho \times E_\rho$, i.e., a fiber of the Albanese fibration of $Y_2$, and $C_2^{(2)}$ be the curve on $Y_2$ defined by the images of the curves 
\[
E_1 = (z, z), \quad E_2 = \left(\rho z-\frac{(1-\rho)}{3}, z\right), \quad E_3 = \left(\rho^2 z-2\frac{(1-\rho)}{3}, z\right)
\]
on $ E_\rho \times E_\rho$. Then $C_1^{(2)} \cap C_2^{(2)}$ is the image on $Y_2$ of the origin in $E_\rho \times E_\rho$. This point is a singular point on $C_2^{(2)}$ of {order} $3$ and is the unique singular point on that curve. Let $X_2$ be the blowup of $Y_2$ at this point. For $i=1, 2$, let $D_i^{(2)}$ be the proper transform of $C_i^{(2)}$ in $X_2$, and define $D_2 = D_1^{(2)} + D_2^{(2)}$. The pair $(X_2, D_2)$ is our second example.

\subsection{Example 3}

Let $Y_3 = Y_2$, $X_3 = X_2$, and $C_2^{(3)} = C_2^{(2)}$ be as in the second example. Consider the fibration $X_3 \to \pp^1$ associated with the first coordinate projection of $E_\rho \times E_\rho$ and let $C_1^{(3)}$ be the fiber passing through the unique singular point of $C_2^{(3)}$. Note that $C_{1}^{(3)}$ is not a multiple fiber of the fibration $X_{3}\rightarrow \pp^{1}$. For $i=1, 2$, let $D_i^{(3)}$ be the proper transform of $C_i^{(3)}$ in $X_3$, and $D_3 = D_1^{(3)} + D_2^{(3)}$. Then $(X_3, D_3)$ is our third example.

\subsection{Example 4}

Let $E_{1-\rho}$ be the quotient of $\cc$ by $\zz[3, 1-\rho] = (1-\rho) \zz[1, \rho]$, consider $E_\rho \times E_{1-\rho}$, and let $Y_4$ be the bielliptic quotient defined by the automorphisms
\[
\varphi_1(w, z) = (\rho w, z + 1), \quad \varphi_2(w, z) = \left(w + \frac{(1-\rho)}{3}, z + \frac{(1-\rho)}{3} \right),
\]
{which have order $3$ and generate an abelian group of order $9$.} Suppose that $C_2^{(4)}$ is the curve on $Y_4$ defined by the images of the curves
\[
E_1 = (z, z), \quad E_2 = (\rho z, z), \quad E_3 = (\rho^2 z, z)
\]
on $E_\rho \times E_{1-\rho}$, and $C_1^{(4)}$ is the fiber of the Albanese map of $Y_{4}$ passing through the unique singular point $p$ of $C_2^{(4)}$. Then, let $X_4$ be the blow up of $Y_{4}$ at the point $p$. For $i=1, 2$, let $D_i^{(4)}$ be the proper transform of $C_i^{(4)}$ in $X_4$. Finally, let $D_4 = D_1^{(4)} + D_2^{(4)}$. The fourth example is the pair $(X_4, D_4)$.

\subsection{Example 5}

We take $Y_5 = Y_4$, $C_2^{(5)} = C_2^{(4)}$, and $X_5 = X_4$. Let $Y_{5}\rightarrow \pp^{1}$ be the fibration associated with the first coordinate projection of $E_{\rho}\times E_{1-\rho}$. Let $C_1^{(5)}$ be the fiber of such a fibration passing through the unique singular point of $C_{2}^{(5)}$. More precisely, $C_{1}^{(5)}$ is the support of a multiple fiber of the fibration $X_{5}\rightarrow \pp^{1}$. For $i=1, 2$, let $D_i^{(5)}$ be the proper transform of $C_i^{(5)}$ in $X_5$. Finally, define $D_5 = D_1^{(5)} + D_2^{(5)}$. Our fifth and final example is the pair $(X_5, D_5)$.

\begin{lemma}\label{alldistinct}
The above examples are mutually distinct.
\end{lemma}
\begin{proof}
We already know from \S \ref{secondthird} that the second and third are distinct, and from \S \ref{fourthfifth} that the fourth and fifth are distinct. For the remaining distinctions, it suffices to compute the first homology groups of the compactifications. Recall that the blow up operation leaves the fundamental group, and hence the first homology group, unchanged. The first example is the blow up at one point of an Abelian surface $Y=E_{\rho}\times E_{\rho}$, so $H_{1}(Y, \zz)=\zz^{4}$. The second and third examples are the blow up of a $\zz/3\zz$ bielliptic surface $Y_{2}$, where $H_{1}(Y_{2}, \zz)=\zz^{2}\oplus\zz/3\zz$ (\cite{Ser90} p.\ 531). On the other hand, the fourth and fifth examples are the blow up of a $\zz/3\zz\times \zz/3\zz$ bielliptic surface $Y_{4}$. In this case, one can compute that $H_{1}(Y_{4}, \zz)=\zz^{2}$ (\cite{Ser90} p.\ 531). We then have that all five examples are distinct.
\end{proof}

\begin{remark}
It is interesting to note that, although the complex hyperbolic surfaces with cusps identified in Example 2 and 3 (or Example 4 and 5) are not isomorphic, they nevertheless have biholomorphic smooth toroidal compactifications. We have recently constructed arbitrarily large families of distinct ball quotients with biholomorphic smooth toroidal compactifications. For more details we refer to \cite{DS15}.
\end{remark}

\section{Proof of Theorem \ref{primo}}\label{thmproof}

We showed above that there are exactly five complex hyperbolic $2$-manifolds of Euler number one that admit a smooth toroidal compactification. It remains to show that these five manifolds are commensurable, i.e., that they all share a common finite-sheeted covering. Since Hirzebruch's ball quotient is arithmetic \cite{Holzapfel}, arithmeticity of the other four examples follows immediately.

Let $\rho = e^{2 \pi i / 3}$, $k = \qq(\rho)$, and $\mathcal{O}_k = \zz[1, \rho]$ be its ring of integers. Holzapfel showed that Hirzebruch's example $(X_0, D_0)$ has fundamental group
\[
\Gamma_0 = \pi_1(X_0 \smallsetminus D_0)
\]
a subgroup of index $72$ in the \emph{Picard modular group} $\Gamma = \PU(2, 1; \mathcal{O}_k)$ associated with a hermitian form on $k^3$ of signature $(2,1)$. Considering the volume of the Picard modular orbifold $\mathcal{H}^2 / \Gamma$, it follows that any subgroup $\Gamma' \subset \Gamma$ of index $72$ determines a quotient of $\mathcal{H}^2$ with Euler--Poincar\'e characteristic one. Consequently, if such a $\Gamma'$ is torsion-free and every parabolic element of $\Gamma'$ is rotation-free, then it defines a complex hyperbolic manifold $\mathcal{H}^2 / \Gamma'$ that admits a smooth toroidal compactification. Since we classified all such complex hyperbolic manifolds above, any $\Gamma'$ with these properties defines one of our five smooth toroidal compactifications.

In \cite{Stover}, the appendix contains eight {nonisomorphic} torsion-free subgroups of the Picard modular group $\Gamma$ of index $72$. In particular, the associated quotients of $\mathcal{H}^2$ are distinct, smooth, and have Euler number one. If we show that exactly five of these subgroups have rotation-free parabolic elements, then these lattices must determine the five smooth toroidal compactifications described in this paper. In particular, the complex hyperbolic manifolds associated with these five surfaces must be commensurable and arithmetic, which proves Theorem \ref{primo}.

The strategy of proof is computational, using the presentation for $\Gamma$ given by Falbel and Parker \cite{FP}. Falbel and Parker showed that $\Gamma$ has a presentation on generators $R$, $P$, and $Q$. Representatives in $\mathrm{GL}_3(\zz[1, \rho])$ for $R, P, Q \in \PU(2, 1)$ are given by:
\begin{align*}
R &= \begin{pmatrix} 0 & 0 & 1 \\ 0 & -1 & 0 \\ 1 & 0 & 0 \end{pmatrix} \\
P &= \begin{pmatrix} 1 & 1 & \rho \\ 0 & \rho & -\rho \\ 0 & 0 & 1 \end{pmatrix} \\
Q &= \begin{pmatrix} 1 & 1 & \rho \\ 0 & -1 & 1 \\ 0 & 0 & 1 \end{pmatrix}
\end{align*}
Moreover, $\mathcal{H}^2 / \Gamma$ has one cusp, and the unique conjugacy class of parabolic subgroups associated with the cusp is represented by $\Delta = \langle P, Q \rangle$, and $\Delta$ fits into an exact sequence
\[
1 \to \zz \to \Delta \to \Delta(2,3,6) \to 1,
\]
where $\Delta(2,3,6)$ is the $(2,3,6)$ triangle group, and
\[
\Delta = \langle P, Q\ |\ (P Q^{-1})^6, P^3 Q^{-2} \rangle.
\]

Given a finite index subgroup $\Gamma' \subset \Gamma$, the conjugacy classes of parabolic subgroups of $\Gamma'$ are then represented (perhaps with repetition) by the groups
\[
\Delta_\sigma = \sigma \Delta \sigma^{-1} \cap \Gamma',
\]
where $\sigma$ runs over all coset representatives of $\Gamma'$ in $\Gamma$. To check that a given $\Delta_\sigma$ contains only rotation-free elements, it suffices to check generators for $\Delta_\sigma$. Indeed, to check that a parabolic group is rotation-free, it suffices to check that its generators in an appropriate basis are strictly upper-triangular (i.e., have all $1$s on the diagonal).

Using Magma \cite{Magma}, we enumerated the eight torsion-free lattices in \cite{Stover}, and calculated generators for a representative of each conjugacy class of parabolic subgroups (a Magma routine that describes these lattices and finds the conjugacy classes of parabolic subgroups was written for \cite{Stover}, and is available on the second author's website). See \S \ref{subsec:Parabolics} for more details. Using the explicit matrices given by Falbel and Parker's generators, we see that exactly five of these lattices determine smooth toroidal compactifications, namely the third, fourth, fifth, seventh, and eighth examples from the appendix to \cite{Stover}. These five manifolds must be the five examples described in this paper, and this completes the proof of Theorem \ref{primo}. \qed

\begin{remark}
Considering $1^{st}$ homology groups, it is clear that the third example in \cite{Stover} is Hirzebruch's example, the fourth and seventh arise from the $\zz / 3 \zz$ bielliptic surface, and the fifth and eighth arise from the $\zz / 3 \zz \times \zz / 3 \zz$ bielliptic {surface}.
\end{remark}

{In Figure \ref{fig:Commensurability} we give the commensurability relations between the five manifolds, where $X_1, \dots X_5$ are the manifolds described in \S \ref{recap}. These were computed with the Magma \cite{Magma} code from \cite{Stover}.}

\begin{figure}
\begin{center}
\begin{tikzpicture}
\node (node1) at (0,0) {$X_1$};
\node (node4) at (4,0) {$X_4$};
\node (node5) at (2,0) {$X_5$};
\node (node2) at (-4,0) {$X_2$};
\node (node3) at (-2,0) {$X_3$};
\node (node24) at (-1,-2) {$X_{2,4}$};
\node (node14) at (4,-2) {$X_{1,4} = X_{1,5} = X_{4,5}$};
\node (node12) at (-5,-2) {$X_{1,2}$};
\node (node13) at (-3,-2) {$X_{1,3}$};
\node (node35) at (1,-2) {$X_{3,5}$};
\node (node25) at (3,-4) {$X_{2,5} = X_{3,4}$};
\node (node23) at (-4,-4) {$X_{2,3}$};
\draw [->] (node12) -- (node1);
\draw [->] (node12) -- (node2);
\draw [->] (node13) -- (node1);
\draw [->] (node13) -- (node3);
\draw [->] (node14) -- (node1);
\draw [->] (node14) -- (node4);
\draw [->] (node14) -- (node5);
\draw [->] (node35) -- (node3);
\draw [->] (node35) -- (node5);
\draw [->] (node24) -- (node2);
\draw [->] (node24) -- (node4);
\draw [->] (node23) -- (node12);
\draw [->] (node23) -- (node13);
\draw [->] (node25) -- (node35);
\draw [->] (node25) -- (node14);
\draw [->] (node25) -- (node24);
\end{tikzpicture}
\caption{Commensurability relations between the five examples. Each arrow represents a $3$-fold covering.}\label{fig:Commensurability}
\end{center}
\end{figure}
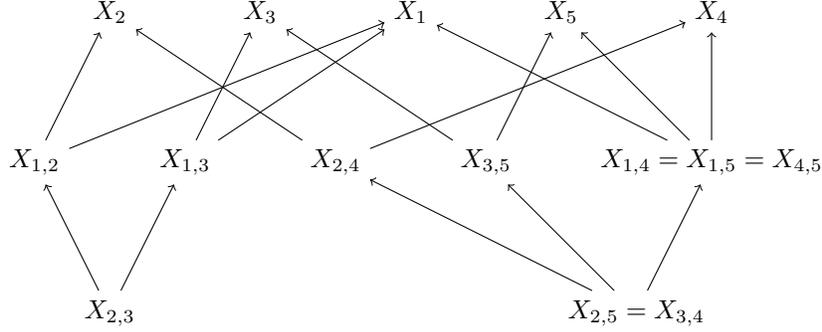

\subsection{More on parabolic subgroups}\label{subsec:Parabolics}

{In this section, we give a few remarks that explicitly connect the cusps of our examples to the group structure of the associated Picard modular group and its unique cusp subgroup.} Retaining the notation from the proof of Theorem \ref{primo}, we consider the following three subgroups of $\Delta$:
\begin{align*}
\Delta_6 &= \langle [Q, P]\, ,\, [P^{-1}, Q] \rangle \\
\Delta_{18, a} &= \langle P^3\, ,\, [Q P Q, P]\, ,\, [Q P^{-1} Q, P] \rangle \\
\Delta_{18, b} &= \langle P^3\, ,\, P Q P Q^{-1} P\, ,\, [Q P Q, P^{-1} Q P] \rangle \\
\Delta_{54} &= \langle [Q P Q, P]\, ,\, P^{-1} (Q P)^2 Q P^{-1} Q^{-1}  \rangle
\end{align*}
Here, $[x, y] = x y x^{-1} y^{-1}$, which we note is the opposite of Magma's notation. The integer part of the subscript denotes the index in $\Delta$, and $\Delta_{18, a}$ is not conjugate to $\Delta_{18, b}$. It is easy to check that each generator is a parabolic with trivial rotational part, which implies that any conjugate in $\PU(2, 1)$ of a $\Delta_i$ is rotation-free.

Also, for each integer $k \ge 1$, define abstract nil $3$-manifold groups
\[
N_k = \left\langle a, b, c\ |\ [a, c]\, ,\, [b, c]\, ,\, [a,b]c^{-k} \right\rangle.
\]
One can check with Magma that
\begin{align*}
\Delta_6 \cong \Delta_{54} &\cong N_1 \\
\Delta_{18, a} \cong \Delta_{18, b} &\cong N_3.
\end{align*}
Suppose that $N_k$ is the maximal parabolic subgroup of a lattice $\Gamma$ in $\PU(2, 1)$, and it is rotation-free. Then $\mathcal{H}^2 / \Gamma$ has a cusp associated with $N_k$, and this cusp can be smoothly compactified by an elliptic curve of self-intersection $-k$ (see \cite[\S 4.2]{HolzBook}). In particular, $\Delta_6$ and $\Delta_{54}$ will be associated with cusps of self-intersection $-1$ and $\Delta_{18, a}, \Delta_{18, b}$ with cusps of self-intersection $-3$.

We now tabulate the conjugacy classes of parabolic subgroups for each of the five examples in this paper. We identify the $\Delta_i$ for which some conjugate of $\Delta_i$ in the Picard modular group appears as a maximal parabolic subgroup of the lattice in $\PU(2, 1)$, but leave it to the reader to calculate the exact conjugates.

\medskip

\begin{center}
\begin{tabular}{| c | c |}
\hline
Compactification & Parabolic subgroups \\
\hline
Hirzebruch's example & \parbox[t]{5cm}{$\Delta_6$ ($3$ distinct conjugacy classes) \\ $\Delta_{54}$ ($1$ conjugacy class)} \\
\hline
\parbox[t]{3cm}{$\zz / 3 \zz$ bielliptic \# 1 \\ (\# 4 in \cite{Stover})} &\parbox[t]{5cm}{$\Delta_{18, a}$ ($1$ conjugacy classes) \\ $\Delta_{54}$ ($1$ conjugacy class)} \\
\hline
\parbox[t]{3cm}{$\zz / 3 \zz$ bielliptic \# 2 \\ (\# 7 in \cite{Stover})} &\parbox[t]{5cm}{$\Delta_{18,a }$ ($1$ conjugacy classes) \\ $\Delta_{54}$ ($1$ conjugacy class)} \\
\hline
\parbox[t]{3.3cm}{$(\zz / 3 \zz)^2$ bielliptic \# 1 \\ (\# 5 in \cite{Stover})} &\parbox[t]{5cm}{$\Delta_{18, b}$ ($1$ conjugacy classes) \\ $\Delta_{54}$ ($1$ conjugacy class)} \\
\hline
\parbox[t]{3.3cm}{$(\zz / 3 \zz)^2$ bielliptic \# 2 \\ (\# 8 in \cite{Stover})} &\parbox[t]{5cm}{$\Delta_{18, b}$ ($1$ conjugacy classes) \\ $\Delta_{54}$ ($1$ conjugacy class)} \\
\hline
\end{tabular}
\end{center}

\medskip

Notice that the sum of the indices is always $72$, since the cusp associated with some $\Delta_i$ is $i$-to-$1$ over the unique cusp of the Picard modular surface and the total covering degree is $72$.


\begin{thebibliography}{ELMNPM}

\bibitem[AMRT10]{Ash} A. Ash, D. Mumford, M. Rapoport, Y.-S. Tai,
Smooth compactifications of locally symmetric varieties. Second
edition. Cambridge Mathematical Library. \textit{Cambridge
University Press, Cambridge}, 2010.

\bibitem[BdF07]{Bd07} G. Bagnera, M. de Franchis, Sur les surfaces hyperelliptiques. \textit{C. R. Acad. Sci.}, \textbf{145} 
(1907), 747--749.

\bibitem[BHPV04]{Bar} W. P. Barth, K. Hulek, C. A. Peters, A. Van de Ven \textit{Compact complex surfaces},
Ergebnisse der Mathematik und ihrer Grenzgebiete. 3. Folge. A Series
of Modern Surveys in Mathematics, 4. Springer-Verlag, Berlin, 2004.

\bibitem[BCP]{BCP} I. Bauer, F. Catanese, R. Pignatelli, Complex surfaces of general type: some recent progress.
\textit{Complex and differential geometry}, Springer Proc. Math. \textbf{8}, Springer, 2011, 1--48.

\bibitem[Bea96]{Bea} A. Beauville, Complex Algebraic Surfaces, Second Edition,
London Mathematical Society Students Texts, 34. \textit{Cambridge University Press, Cambridge}, 1996.

\bibitem[Bel14]{Belo} M. Belolipetsky, Hyperbolic orbifolds of small volume. International Congress of Mathematicians. Vol. II, 837--851, Kyung Moon Sa, 2014.


\bibitem[BJ06]{Borel2} A. Borel, L. Ji, Compactifications of locally symmetric spaces. Mathematics: Theory \& Applications,
\textit{Birk\"auser Boston, Inc. Boston, MA}, 2006.

\bibitem[CS10]{Ste} D. I. Cartwright, T. Steger, Enumeration of the $50$ fake projective planes.
\textit{C. R. Math. Acad. Sci. Paris}, \textbf{258} (2010), no. 8,
2708--2713.

\bibitem[Deb99]{Deb} O. Debarre, \textit{Tores et vari\'et\'es ab\'eliennes complexes},
Cours Sp\'ecializ\'es, 6. Soci\'et\'e Math\'ematique de France,
Paris; EDP Sciences, Les Ulis, 1999.

\bibitem[DM93]{Deligne} P. Deligne, G. Mostow, Commensurabilities among lattices in $PU(1,n)$, Annals of Mathematics Studies 132, \textit{Princeton University Press, Princeton, NJ}, 1993.

\bibitem[DiC15]{DiCerbo2} G. Di Cerbo, L. F. Di Cerbo, Effective results for complex hyperbolic manifolds. \textit{J. London Math. Soc.}, \textbf{91} (2015), no. 1, 89--104.

\bibitem[DiC12]{Luca1} L. F. Di Cerbo, Finite-volume complex-hyperbolic surfaces, their toroidal compactifications, and geometric applications.
\textit{Pacific J. Math.}, \textbf{255} (2012), no. 2, 305--315.

\bibitem[DiC13]{DiC13} L. F. Di Cerbo, On the classification of toroidal compactifications with $3\overline{c}_{2}=\overline{c}^{2}_{1}$ and $\overline{c}_{2}=1$. arXiv:1309.5516v3[math.AG].


\bibitem[DS16]{DS15} L. F. Di Cerbo, M. Stover, Multiple realizations of varieties as ball quotient compactifications.
\textit{Michigan Math. J.} \textbf{65} (2016), no. 2, 441--447.

\bibitem[DS17]{DS16} L. F. Di Cerbo, M. Stover, Bielliptic ball quotient compactifications and lattices in $\mathrm{PU}(2,1)$ with finitely generated commutator subgroup.
\textit{Ann. Inst. Fourier}, \textbf{67} (2017), no. 1, 315--328.

\bibitem[FP06]{FP} E. Falbel and J. Parker, The geometry of the Eisenstein--Picard modular group.
\textit{Duke Math. J.}, \textbf{131} (2006), no. 2, 249--289.

\bibitem[Fri98]{Friedman} R. Friedman, \textit{Algebraic Surfaces and Holomorphic Vector Bundles},
Univesitext. Springer-Verlag, New York 1998.

\bibitem[GMM]{GMM} D. Gabai, R. Meyerhoff, P. Milley, Minimum volume cusped hyperbolic three-manifolds.
\textit{J. Amer. Math. Soc.} \textbf{22} (2009), no. 4, 1157--1215.

\bibitem[Gro82]{Gro1} M. Gromov, Volume and bounded cohomology.
\textit{Publ. Math. Inst. Hautes E\'tudes Sci.}, \textbf{56} (1982),
5--99.

\bibitem[Har77]{Har} R. Hartshorne, \textit{Algebraic Geometry}, Springer-Verlag, New York 1977, Graduate Texts In Mathematics, No. 52.

\bibitem[Hir84]{Hir84} F. Hirzebruch, Chern numbers of algebraic surfaces: an example. \textit{Math. Ann.}, \textbf{266} 
(1984), 351--356.

\bibitem[Hol86]{Holzapfel} R. P. Holzapfel, Chern numbers of algebraic surfaces--Hirzebruch's examples are Picard modular surfaces.
\textit{Math. Nachr.}, \textbf{126} (1986), 255--273.

\bibitem[Hol98]{HolzBook} R. P. Holzapfel, Ball and surface arithmetics. Aspects of Mathematics, \textbf{29}. Friedr. Vieweg \& Sohn, 1998.

\bibitem[Hol04]{Holz04} R. P. Holzapfel, Complex hyperbolic surfaces of Abelian type.
\textit{Serdica Math. J.} \textbf{30} (2004), no. 2-3, 207--238.

\bibitem[Hum98]{Hummel} C. Hummel, Rank one lattices whose parabolic isometries have no rotational part.
\textit{Proc. Am. Math. Soc.}, \textbf{126} (1998), 2453--2458.

\bibitem[Kaw78]{Kaw} Y. Kawamata, On deformations of compactifiable complex manifolds.
\textit{Math. Ann.}, \textbf{235} (1978), no. 3, 247--265.

\bibitem[Kli03]{Kli} B. Klingler, Sur la rigidit\'e de certains groupes foundamentaux, l'arithm\'eticit\'e des r\'eseaux hyperboliques
complexes, et le ``faux plans projectifs''. \textit{Invent. Math.},
\textbf{153} (2003), no. 1, 105--143.

\bibitem[Mag]{Magma} W. Bosma, J. Cannon, C. Playoust, The {M}agma algebra system. {I}. {T}he user language. \textit{J. Symbolic Comput.}, Computational algebra and number theory (London, 1993), \textbf{24} (1997), no. 3-4, 235--265.


\bibitem[Miy77]{Miyaoka} Y. Miyaoka, On the Chern numbers of surfaces of general type.
\textit{Invent. Math.}, \textbf{42} (1977), 225--237.

\bibitem[Mok12]{Mok} N. Mok, Projective algebraicity of minimal compactifications of complex-hyperbolic space forms of finite-volume.
Perspective in analysis, geometry, and topology, 331-354,
\emph{Prog. Math.,} \textbf{296}, \emph{Birkh\"auser/Springer, New
York}, 2012.

\bibitem[Mo08]{Momot} A. Momot, Irregular ball-quotient surfaces with non-positive Kodaira dimension.
\textit{Math. Res. Lett.} \textbf{15} (2008), no. 6, 1187--1195.

\bibitem[Mum77]{Mumford} D. Mumford, Hirzebruch's Proportionality Theorem in the Non-Compact Case.
\textit{Invent. Math.}, \textbf{42} (1977), 239--272.

\bibitem[Mum79]{Mum} D. Mumford, An Algebraic surface with $K$ ample, $(K^{2})=9$, $p_{g}=q=0$.
\textit{Amer. J. Math.}, \textbf{101} (1979), no.1, 233--244.


\bibitem[PY07]{Pra} G. Prasad, S.-K. Yeung, Fake projective planes.
\textit{Invent. Math.}, \textbf{168} (2007), no. 2, 321--370.


\bibitem[Sak80]{Sakai} F. Sakai, Semistable curves on algebraic surfaces and logarithmic pluricanonical maps.
\textit{Math. Ann.}, \textbf{254} (1980), no. 2, 89--120.

\bibitem[Ser90]{Ser90} F. Serrano, Divisors on Bielliptic Surfaces and Embeddings in $\pp^{4}$. \textit{Math. Z.}, \textbf{203} 
(1990), 527--533.


\bibitem[Sto11]{Stover} M. Stover, Volumes of Picard modular surfaces.
\textit{Proc. Amer. Math. Soc.}, \textbf{139} (2011), no. 9,
3045--3056.

\bibitem[TY87]{TianY} G. Tian, S.-T. Yau, Existence of K\"ahler-Einstein metrics on complete K\"ahler manifolds
and their applications to algebraic geometry. \textit{Mathematical
aspects of string theory (San Diego, Calif., 1986)}, 574-628, Adv.
Ser. Math. Phys., 1, \emph{World Sci. Publishing, Singapore}, 1987.

\bibitem[Urz]{Urzua} G. Urz{\'u}a, Arrangements of curves and algebraic surfaces.
\textit{J. Algebraic Geom.}, \textbf{19} (2010), no. 2, 335--365.

\bibitem[Wan72]{Wan} H. C. Wang, \textit{Topics on totally discontinuous groups}, Symmetric Spaces, 459-487,
edited by W. B. Boothby and G. L. Weiss, Pure and Appl. Math. 8,
Marcel Dekker, New York, 1972.

\bibitem[Yau78]{Yau} S.-T. Yau, On the Ricci curvature of a compact K\"ahler manifold and the complex Monge-Amp\`ere equation.
\textit{Comm. Pure Appl. Math.}, \textbf{31} (1978), 339--411.

\bibitem[Yeu04]{Yeu} S.-K. Yeung, Integrality and arithmeticity of co-compact lattice corresponding
to certain complex two-ball quotients of Picard number one.
\textit{Asian J. Math.}, \textbf{8} (2004), no. 21, 107--129.


\end{thebibliography}
\end{document}